\documentclass{article}

\usepackage{a4wide}
\usepackage{cite}
\usepackage[utf8]{inputenc}
\usepackage[T1]{fontenc}

\usepackage[colorlinks=true, pdfstartview=FitV, linkcolor=blue,citecolor=blue, urlcolor=blue]{hyperref}

\usepackage[french,english]{babel}
\usepackage{amsmath}
\usepackage{amscd}
\usepackage{amssymb}
\usepackage{amsrefs}
\usepackage{amsthm}
\usepackage{cases}
\usepackage{framed}
\usepackage{graphicx}
\usepackage{latexsym}
\usepackage{color}
\usepackage{array,multirow,makecell}
\usepackage[capitalise]{cleveref}
\usepackage{enumerate}
\usepackage{todonotes}
\usepackage{enumitem} 
\usepackage{bbm}

\definecolor{purple}{RGB}{100, 0, 200}

\setcellgapes{1pt}
\makegapedcells

\theoremstyle{plain} 
\newtheorem{thm}{Theorem}[section]
\newtheorem{prop}[thm]{Proposition}
\newtheorem{lem}[thm]{Lemma}

\theoremstyle{definition} 
\newtheorem{df}[thm]{Definition}

\theoremstyle{remark} 
\newtheorem{rmk}[thm]{Remark}

\newcommand{\R}{\mathbb{R}}

\newcommand{\p}{\partial}
\newcommand{\be}{\begin{equation}}
\newcommand{\ee}{\end{equation}}
\newcommand{\ba}{\begin{aligned}}
\newcommand{\ea}{\end{aligned}}

\renewcommand{\ng}{n_\gamma}
\newcommand{\pg}{p_\gamma}

\newcommand{\cg}{c_\gamma}
\newcommand{\Ng}{N_\gamma}
\newcommand{\Pg}{P_\gamma}
\newcommand{\Lg}{L_\gamma}
\newcommand{\Papp}{P_\text{app}}

\newcommand{\cD}{\mathcal D_\gamma}
\newcommand{\xim}{\xi^-_\gamma}
\newcommand{\txg}{\tilde \xi_\gamma}
\newcommand{\xiint}{\xi^\text{int}_\gamma}

\title{Traveling waves for the porous medium equation in the incompressible limit:\\ asymptotic behavior and nonlinear stability
}

\author{Anne-Laure Dalibard\footnote{Sorbonne Université, Université Paris-Diderot SPC, CNRS,  Laboratoire Jacques-Louis Lions, LJLL, F-75005 Paris; dalibard@ann.jussieu.fr}, Gabriela Lopez-Ruiz\footnote{Sorbonne Université, Université Paris-Diderot SPC, CNRS,  Laboratoire Jacques-Louis Lions, LJLL, F-75005 Paris; lopezruiz@ljll.math.upmc.fr} ~and Charlotte Perrin\footnote{Aix Marseille Univ, CNRS, Centrale Marseille, I2M, Marseille, France; charlotte.perrin@univ-amu.fr} }

\begin{document}

\maketitle

\begin{small}
\begin{abstract}
In this study, we analyze the behavior of monotone traveling waves of a one-dimensional porous medium equation modeling mechanical properties of living tissues.
We are interested in the asymptotics where the pressure, which governs the diffusion process and limits the creation of new cells, becomes very stiff, and the porous medium equation degenerates towards a free boundary problem of Hele-Shaw type.
This is the so-called \emph{incompressible limit}.
The solutions of the limit Hele-Shaw problem then couple ``free dynamics'' with zero pressure, and ``incompressible dynamics'' with positive pressure and constant density. 
In the first part of the work, we provide a refined description of the traveling waves for the porous medium equation in the vicinity of the transition between the free domain and the incompressible domain.
The second part of the study is devoted to the analysis of the stability of the traveling waves. 
We prove that the linearized system enjoys a spectral gap property in suitable weighted $L^2$ spaces, and we give quantitative estimates on the rate of decay of solutions.
The nonlinear terms are treated perturbatively, using an $L^\infty$ control stemming from the maximum principle.
As a consequence, we prove that traveling waves are stable under small perturbations.
\end{abstract}

	\bigskip
	\noindent{\bf Keywords:} Porous Medium equation, traveling waves, incompressible limit, mesa limit, stability,\\ Hele-Shaw equations.
	
	\medskip
	\noindent{\bf MSC:} 35C07, 35K57, 35B40, 35B35.
\end{small}

\section{Introduction}
\label{sec:intro}

This paper is devoted to the asymptotic analysis and the stability of traveling waves for the porous medium equation (PME). More precisely, let us consider the following nonlinear parabolic equation
\begin{equation}\label{eq:pde-0}
\partial_t n - \partial_x \big(n \partial_x p(n) \big) = n \Phi(p(n)),
\end{equation}
endowed with the boundary conditions
\[
\lim_{x\to \pm \infty} n(t,x)=n_\pm,
\]
where $n_\pm$ are constant stationary states of the equation.
This equation has been introduced in the literature to model tissue growth and, particularly, in the propagation of tumors (see for instance \cites{Perthame2015,kim_požár_2017,David2021}).
The left-hand side corresponds to the Porous Medium Equation (PME): the density of cells, $n$, is transported by a velocity given by the Darcy law $v=-\partial_x p$ where $p = p(n)$ denotes the mechanical pressure.
The right-hand side models the cell proliferation in the medium, proliferation which is limited by the pressure.
Hence, the function $\Phi$ is usually taken as a decreasing function of the pressure and is such that $\Phi(p_M) = 0$ for some $p_M>0$ called the \emph{homeostatic pressure}.
In this study, we shall assume for simplicity that
\begin{equation} \label{eq:P-G}
p(n) = \pg(n) = n^\gamma~ \text{with}~\gamma > 1, \qquad  \Phi(p) = 1-p. 
\end{equation}
In other words, the function $\Phi$ becomes negative above the threshold pressure $p_M = p_\gamma(n_M) = 1$, which means that cells are destroyed above the \emph{maximal packing density} $n_M=1$.
We will also pick $n_-=n_M=1$, and $n_+=0$.

\bigskip
This study aims to analyze the behavior of traveling waves (TWs) solutions of~\eqref{eq:pde-0} when the parameter $\gamma$ appearing in the equation of state~\eqref{eq:P-G} tends to $+\infty$. 
For $\Phi(p) = 0$, {\it i.e} without the reaction term in the equation, this limit $\gamma \to +\infty$ is referred as the \emph{mesa limit} and has been studied for instance by Caffarelli and Friedman~\cite{caffarelli1987}.
In this paper, the authors consider an initial datum larger than $1$ on a nontrivial set and show that this upper part exceeding $1$ collapses at $t=0^+$ to $\{n=1\}$. This phenomenon is due to the blow up of the diffusivity $np'_\gamma(n) = \gamma n^\gamma \to +\infty$ when $ n > 1$.
The singular limit $\gamma \to +\infty$ for solutions of the PME is then called the ``mesa'' limit in reference to the shape of the target density $n_\infty \in [0,1]$ which is similar to the flat-topped mountains.
In the presence of a growth source term $\Phi$, the limit $\gamma \to +\infty$ has been first tackled by Perthame {\it et al.} in~\cite{PQV_2014}. 
As in the previous case, the blow-up of the pressure as $\gamma\to +\infty$ when $n>1$ forces the limit density to lie in $[0,1]$. 
The sequence $(n_\gamma)_{\gamma>1}$ of weak solutions to~\eqref{eq:pde-0} is then shown to converge (for a suitable topology) towards a weak solution of the following Hele-Shaw system
\begin{subnumcases}{\label{eq:HS}}
\partial_t n - \partial_x(n \partial_x p) = n \Phi(p),\label{eq:HS-mass} \\
0 \leq n \leq 1, \quad (1-n)p = 0, \quad p \geq 0, \label{exclusion}\\
p \ \big(\partial^2_x p + \Phi(p) \big) = 0.\label{eq:HS-compl}
\end{subnumcases}
The transition between equation~\eqref{eq:pde-0} and system~\eqref{eq:HS} is usually called the \emph{incompressible limit} in reference to the fact that, when the solution $n$ of~\eqref{eq:HS} reaches $1$, it is blocked to this maximal value (the combination of the mass equation~\eqref{eq:HS-mass} with the complementary relation~\eqref{eq:HS-compl} yields formally $\partial_t n = 0$ in $\{n=1\}$) and the medium cannot be further compressed. 
Beyond the physical and biological relevancy of system~\eqref{eq:pde-0} seen as an approximation of~\eqref{eq:HS}, Mellet {\it et al.}~\cite{Mellet2017} have shown that the incompressible limit can provide crucial qualitative information on the solutions of the Hele-Shaw system~\eqref{eq:HS}, like the regularity of the free boundary $\partial \{n=1\}$.\\
To finish with the incompressible limit, let us mention that this type of singular limit has been studied in other frameworks: for other singular equations of state~\cite{hecht_vauchelet_2017}, in the case of coupling with the dynamics of nutrients~\cite{David2021}, in the case of more than one type of cancerous cell as seen in \cites{Bubba2020, Debiec2020, degond2020}, when the Darcy law is replaced by the Brinkman equation~\cite{Perthame2015} or the Navier-Stokes equations~\cite{Vauchelet2017}.

\bigskip
Up to our knowledge, the issue of TWs solutions to~\eqref{eq:HS} remains rare in the literature (see~\cite{perthame2014} when nutrients are considered), even when the topic was intensively studied for nonlinear reaction-diffusion equations like~\eqref{eq:pde-0}. 
Indeed, TWs as a class of special solutions have been shown to provide valuable information on general solutions of these reaction-diffusion equations (see the books~\cite{volpert1994} and~\cite{gilding2012}).
Most of the results concern the long-term behavior (convergence to TWs, asymptotic rate of propagation of disturbances) or the behavior close to interfaces of general solutions.\\
Regarding the issue of interfaces, Gilding and Kersner study in~\cite{gilding2005fisher} the existence of \emph{sharp} (or \emph{finite}) TWs whose support is bounded on one side in case of nonlinear degenerate diffusion, and deduce a result about the existence of an interface $\partial \{n=0\}$ for general solutions. 
In~\cite{galaktionov1999}, TWs are used to study the regularity of the general solutions near the free boundary $\partial \{n=0\}$, as well as for the derivation of the interface motion.
The essential tools of the analysis are then: 
the continuity of the flux across the interface and a comparison principle bracketing a general solution between two TWs.\\
Concerning the long-time behavior of solutions to reaction-diffusion scalar equations like~\eqref{eq:pde-0}, let us mention two types of results related to the nature of the wave-front.
For \emph{sharp fronts}, that is, TWs with support bounded from above (or below), Kamin and Rosenau prove in~\cite{kamin2004} that initial data decaying sufficiently fast at infinity converge (in a specific sense) towards a sharp TW. 
The techniques they employ are inspired by $L^1$-stability theory of shock waves for viscous conservation laws (see for instance~\cite{serre2002}): use of comparison principle (already mentioned above), derivation of $L^1$ conservation, and contraction principles with an exponential weight. 
It is worth pointing out that this result cannot be extended to \emph{smooth fronts}, {\it i.e.} TWs that do not vanish and remain smooth on $\R$.
Indeed the weight used in~\cite{kamin2004} is specific to the critical speed $c^*$ at which the sharp fronts travel (see Theorem~\ref{t:gilding} below) and is not suited for the smooth fronts propagating at speed $c> c^*$.
To our knowledge, the only result dealing with smooth fronts is a spectral stability result obtained recently by Leyva and Plaza in~\cite{Leyva2020}. 
In their work, the difficulties associated with the degeneracy of the diffusion term are overcome with the derivation of a kind relative entropy estimate with a well-suited exponential weight.  

\bigskip
In this paper, the study of smooth TWs of~\eqref{eq:pde-0} as $\gamma \to +\infty$ can be seen as a first step in the analysis of the free boundary $\partial\{n=1\}$ for the limit Hele-Shaw system~\eqref{eq:HS}.
Our contributions are twofold:
we first give a qualitative and quantitative description (in terms of $\gamma$) of smooth TWs of~\eqref{eq:pde-0} and show the convergence towards TWs of~\eqref{eq:HS} that are discontinuous at the interface $\partial\{n=1\}$; we also study the nonlinear asymptotic stability of the smooth TWs for small (quantified in terms of $\gamma$) general perturbations of these wave-fronts.\\
As in~\cite{galaktionov1999}, our analysis relies strongly on the control of the flux around the interface (passage to the limit as $\gamma \to +\infty$, determination of the transmission conditions across the interface on the limit system); and the comparison principle (quantitative behavior of TWs as $\gamma \to +\infty$, control of general solutions lying between two TWs).
Compared to the stability analysis of Leyva and Plaza~\cite{Leyva2020}, we have to deal with additional nonlinear contributions that we treat in a perturbative manner and control thanks to a Poincar\'e-type inequality.
This latter also allows us to get a decay rate of the perturbation as $t \to +\infty$.

\subsection*{Statement of main results}
In this paper, we focus on traveling waves solutions of~\eqref{eq:pde-0}-\eqref{eq:P-G}, that is solutions $n_\gamma$ such that $n_\gamma(t,x) = \Ng(x-ct)$ where $\Ng$ is the wave profile, $\xi =x -ct$ is the wave coordinate and c is the speed of propagation of the wave.
The profile $\Ng$ is then solution to the differential equation:
\begin{equation}\label{eq:Ng}
-c \Ng' - \gamma(\Ng^\gamma \Ng')' = \Ng(1-\Ng^\gamma).
\end{equation}
The above equation admits two equilibrium states: $N \equiv 0$ (unstable) and $N\equiv 1$ (stable), and we seek therefore wavefronts $\Ng$ connecting these two states:
\begin{equation}\label{eq:Ng-endpt}
\lim_{\xi \to -\infty} \Ng(\xi) = 1, \qquad \lim_{\xi \to +\infty} \Ng(\xi) = 0.
\end{equation}
The existence and uniqueness (up to a shift) of a monotone (decreasing) solution to~\eqref{eq:Ng}-\eqref{eq:Ng-endpt}, as well as the asymptotic behavior of $\Ng$ close to $\pm \infty$, were previously investigated by Gilding and Kersner~\cite{gilding2005fisher} for $c$ larger than a threshold velocity $c^*_\gamma > 0$ (see below Theorem~\ref{t:gilding} for a precise statement).
In the present study, we intend to analyze further the behavior of $\Ng$ and $\Pg(\xi) = (\Ng(\xi))^\gamma$, the associated pressure profile, with respect to the parameter $\gamma$.
Our first main result concerns the qualitative and quantitative behaviors as $\gamma \to +\infty$.

\begin{thm}\label{thm:prop-Ng}
Let $\gamma > 1$ sufficiently large, $c > 1$ be fixed, independent of $\gamma$, and let $\Ng$ be the solution of~\eqref{eq:Ng}-\eqref{eq:Ng-endpt} such that $\Pg(0) = \frac{1}{\gamma}$.	
Then the following properties hold true.
\begin{itemize}
	\item There exist $\xi^-_\gamma,\tilde{\xi}_\gamma$ with $\xi^-_\gamma = O\left(\frac{1}{\sqrt{\gamma}}\right) < 0 < \tilde{\xi}_\gamma = O\left(\frac{1}{\gamma}\right)$, such that the profile $(\Ng,\Pg)$ satisfies
	\begin{itemize}
		\item in the \emph{congested zone} $\xi < \xi^-_\gamma$, the density $\Ng$ converges uniformly to $1$: there exists a constant $C>0$ depending only on $c$ such that
		\begin{equation}\label{encadrement-N-congest}
		 \left(\dfrac{C}{\sqrt{\gamma}}\right)^{\frac{1}{\gamma}} \leq \Ng(\xi) \leq 1 \qquad \forall \ \xi \leq \xi^-_\gamma,
		\end{equation}
		and there exist constants $C'\geq C > 0$ independent of $\gamma$ such that
		\begin{equation}\label{encadr-P-xim}
		1-\left(1-\frac{C'}{\sqrt{\gamma}}\right)e^{(1-C\gamma^{-1/2}) \xi}
		\leq \Pg(\xi)
		\leq 1-\left(1-\frac{C}{\sqrt{\gamma}}\right)e^{ \xi} \qquad \forall \ \xi \leq \xi^-_\gamma;
		\end{equation}
		
		\item in the \emph{intermediate region} $ \xi \in [\xi^-_\gamma, \tilde{\xi}_\gamma]$, $\Ng'$ takes exponentially large values with respect to $\gamma$:
		\begin{equation}\label{eq:explo-Np}
		\|\Ng'\|_{L^\infty(\xi^-_\gamma, \tilde{\xi}_\gamma)}
		= O\left(\left(1-\frac{1}{2c}\right)^{-\gamma}\right),
		\end{equation}
		while the pressure $\Pg$ converges uniformly to $0$ as $\gamma \to +\infty$: there exists $\delta \in (0, 1-c^{-1})$, independent of $\gamma$ such that
		\begin{equation}
		   \left(1-\frac{1}{c}-\delta\right)^{\gamma} \leq \Pg(\xi) \leq \dfrac{C}{\sqrt{\gamma}}\qquad \forall \ \xi \in [\xi^-_\gamma,\tilde{\xi}_\gamma];
		\end{equation}
		
		\item in the \emph{free zone} $ \xi > \tilde{\xi}_\gamma$, the pressure $\Pg$ takes exponentially small values (wrt $\gamma$): $\Pg(\xi) \leq \left(1-\frac{1}{2c}\right)^{\gamma} $ and $\Ng$ decreases exponentially to $0$ as $\xi \to +\infty$: there exists $\delta > 0$ independent of $\gamma$, such that for $\gamma$ large enough
		\begin{equation}\label{encadr-N-free}
		\left(1-\frac{1}{c}-\delta\right) \exp\left(-\left(\frac{1}{c} + \delta\right)\xi\right)
		\leq \Ng(\xi) 
		\leq \left(1-\frac{1}{c} + \delta\right) \exp\left(- \frac{1}{2c} \xi\right)
		\quad \forall \ \xi > \tilde \xi_\gamma;
		\end{equation}
	\end{itemize}
	
	\item As $\gamma \to +\infty$, there exists $(N_{HS}, P_{HS}) \in L^\infty(\R) \times W^{1,\infty}(\R)$ such that $\Ng \to  N_{HS}$ in $L^p_\text{loc}(\R)$ and  $\Pg \rightarrow P_{HS} $ in $W^{1,p}_\text{loc}(\R)$ for any $p\in [1, \infty[$, and $(N_{HS}, P_{HS})$ is a wave-front profile of the Hele-Shaw equations~\eqref{eq:HS} such that $P_{HS} (\xi) = (1-e^\xi)\mathbf{1}_{\xi \leq 0}$, $\lim_{\xi \to 0^+}  N_{HS} = 1-\frac{1}{c}$.
\end{itemize}
\end{thm}

\medskip
\begin{rmk}
Concerning the convergence of $(\Ng, \Pg)$ towards $(N_{HS}, P_{HS})$, a key ingredient of our proof is the uniform control of the flux $J_\gamma = c\Ng + \Ng \Pg'$ which is such $J_\gamma' = -\Ng(1-\Pg) \in [-1,0]$.
The control of $J_\gamma$ implies in particular the control of $\Pg'$ and thus yields the uniform convergence of $(\Pg)_\gamma$.
It is important to note that this uniform convergence of $(\Pg)$ is uncorrelated to the convergence of $(\Ng)_\gamma$. 
Indeed, we have $\Pg' = \gamma \Ng^{\gamma-1} \Ng'$ but the pre-factor $\gamma \Ng^{\gamma-1}$ which tends to $0$ on a half-space, prevents us to get a uniform bound on $\Ng'$. 
Actually this derivative blows up as it can be observed on~\eqref{eq:explo-Np}.
The uniform convergence of the flux $J_\gamma$ is also crucial to determine the value of $N_{HS}$ on the right side of the interface $\xi=0$. 
Since then $\displaystyle J_{HS}(0)=c +\lim_{\xi \to 0^-} P_{HS}'(\xi) = c-1$, we deduce that $\lim_{\xi \to 0^+}  N_{HS} = c^{-1} J_{HS}(0) = 1-\frac{1}{c}$.
\end{rmk}

\medskip

\begin{rmk}
A legitimate question is the possible extension of the previous result to more general pressure laws (as for instance the singular potentials considered in~\cite{hecht_vauchelet_2017} or~\cite{dalibard2020}) and reaction terms $\Phi$. 
Our analysis actually starts with the results obtained by Gilding and Kersner~\cite{gilding2005fisher}. 
In particular in~\cite{gilding2005fisher}, the determination of the critical speed $c^* = c^*_\gamma$ is specific to the pressure law $p_\gamma(n) = n^\gamma$. 
To our knowledge, the explicit characterization of $c^*$ has not been tackled in the literature, more precisely we would need an upper bound on $c^*$ independent of the parameter characterizing the incompressible limit.
The extension of~\cite{gilding2005fisher} to the case of more general pressures and reaction terms is therefore out of the scope of the present paper but there is a reasonable hope for a generalization of the previous theorem once the existence of a profile $\Ng$ for a fixed speed $c > c^*$ (independent of parameter $\gamma$) is ensured.

We believe that several steps of our strategy could be extended to other pressure laws (analysis of the phase portrait of the traveling wave and consequences, design of appropriate weights for the coercivity of the linearized operator, etc.)
However, in several instances some quantitative arguments rely heavily on  fine properties of $\Ng$ (e.g. the description of the transition zone).
It is unavoidable that such properties will depend on the exact nature of the pressure law, and that a case by case analysis needs to be performed.

\end{rmk}

\bigskip

Our second result is dedicated to the analysis of stability of the wavefront $\Ng$ in weighted Sobolev spaces. To that end, we introduce the weight
\[
W(\xi):=\Ng(\xi)^\gamma \exp\left(\int_{\xi_\gamma^-}^\xi \frac{c}{\gamma  \Ng^\gamma}\right).
\]
Note that $W$ has a double exponential growth as $\xi\to+\infty $, and a (slow) exponential decay as $\xi \to-\infty$. Therefore, $W$ will provide a very good control of the difference $\ng(t,x)-\Ng(x-ct)$ in the free zone $x-ct>0$.

Our result is the following:


\begin{thm}\label{thm:stab-Ng}
There exists constants $\eta_1, \eta_2\in]0,1[$, depending only on $c>1$, such that the following result holds.

Let $\gamma > 1$ be fixed, sufficiently large.
We make the following assumptions on the initial data $n^0_\gamma$:
\begin{enumerate}[label=(H\arabic*)]
\item  $n^0_\gamma$ lies between two shifts of $\Ng$, {\it i.e.} there exists $h >0$ such that $n^0_\gamma(x) \in [\Ng(x+h),\Ng(x-h)]$ for all $x\in \R$;

\item The difference $n^0_\gamma-\Ng$ is sufficiently decaying, namely
\[
\int_\R \left({n^0_\gamma(x)- \Ng(x)}\right)^2 W(x)  dx  <\infty.
\]
\end{enumerate}
Let $n_\gamma$ be the solution of~\eqref{eq:pde-0} associated with $n^0_\gamma$.

Then there exists a constant $c_\gamma>0$, $c_\gamma=O(\eta_1^\gamma)$, such that if $|h|\leq \eta_2^\gamma$, the following inequality holds:
\[
    \int_\R (n_\gamma(t,x) - \Ng(x-ct))^2 W(x-ct)
\ dx \\
\leq e^{-c_\gamma t} \int_\R \left({n^0_\gamma(x)- \Ng(x)}\right)^2 W(x) dx \qquad \forall t \geq 0.
\]

Moreover, setting $u_\gamma(t,x) := (n_\gamma(t,x) - \Ng(x-ct))/\Ng'(x-ct)$, we have the additional dissipation of energy:
\[
\gamma \int_0^\infty \int_{\R}   (\p_x u_\gamma(t,x))^2
(\Ng^{\gamma} (\Ng')^2 W)(x-ct) 
\ dx \:dt \leq  \int_\R \left({n^0_\gamma(x)- \Ng(x)}\right)^2 W(x) dx.
\]

\end{thm}

Let us give a short sketch of proof of the above result. An important feature of equation \eqref{eq:pde-0} lies in the fact that its linearization around $\Ng(x-ct)$ is spectrally stable in suitable weighted Sobolev spaces. 
This property has been identified recently by Leyva and Plaza \cite{Leyva2020}, using Sobolev spaces with an exponential weight. 
Here, we work with different weights, which we believe follow more closely the structure of the equation, see Lemma \ref{prop:lin} and subsection \ref{ssec:linearized}, and which give a better control in the congested zone.
One crucial point of our analysis lies in the derivation of a new weighted Poincaré inequality associated with this weight, see Proposition \ref{prop:poincare}.
This allows us to prove a spectral gap property, leading to the exponential decay announced in the above Theorem.
Once the dissipation properties of the linearized equation have been identified and quantified, we perform the nonlinear estimates by treating the quadratic terms as perturbations. 
In this regard, assumption (H1) allows us to have a uniform $L^\infty$ control on $\ng(t,x)-\Ng(x-ct)$, thanks to the parabolic nature of the equation. 

Note that the rate of decay $c_\gamma$ of the energy is exponentially small. 
This is linked to the exponential blow-up of $\Ng'$ in the transition zone, see Theorem \ref{thm:prop-Ng}.
This blow-up also imposes a strong limitation on the admissible size of the  perturbation in $L^\infty$, and thereby on the size of $h$.
It is not clear whether this assumption could be substantially lowered, taking for instance initial perturbations that would be algebraically - but not exponentially - small. Indeed, it is possible that the strong variations of $\Ng$ in the transition zone destabilize the flow.

\vspace{1cm}
Our study is organized as follows. 
In Section~\ref{sec:TW-descript}, we describe traveling fronts for both systems~\eqref{eq:pde-0} and~\eqref{eq:HS} and give a refined behavior of the profile $\Ng$ in the transition zone between the congested region and the free region.
Next, we prove in Section~\ref{sec:stability} the asymptotic stability of the profile $\Ng$ ($\gamma$ being fixed) for some $L^2$-weighted norm.
Finally, we have postponed in Section~\ref{sec:appendix} the proofs of some technical lemmas used in Section~\ref{sec:stability}.

\section{Traveling waves for the Hele-Shaw system and the porous media equation}{\label{sec:TW-descript}}
This section is devoted to studying the existence and properties of traveling fronts of both systems: Hele-Shaw and the mechanical model of tumor growth with ``stiff pressure law'' depending on the parameter $\gamma$. For the latter, an asymptotic expansion of this type of solution will be computed.

\subsection{TW for the limit Hele-Shaw system}

We look for traveling wave-type solutions of the form $(n,p)=(N,P)(x-ct)$, where $c>0$ is a constant representing the traveling wave speed and  $N, P$ are real nonnegative functions. We may assume that $c> 0$, since for $c= 0$ we find again the stationary solutions, and the case $c<0$ can be reduced to $c>0$ by reflection. 

\begin{lem}\label{tw_HS}
Let $c>1$ be arbitrary, and let $\xi$ denote the traveling wave variable $\xi=x-ct$.

\begin{enumerate}
    \item Define the profile $(N_{HS}, P_{HS})\in L^\infty(\R)\times W^{1,\infty}(\R)$ by
    \be\label{def:barN-P}
    P_{HS}(\xi)=\left\{
    \begin{array}{ll}
    0&\text{ if } \xi>0,\\
    1-e^\xi&\text{ if } \xi<0,\end{array}
    \right.
    \qquad
     N_{HS}(\xi)=\left\{
    \begin{array}{ll}
   \displaystyle\left( 1-\frac1c\right) e^{-\frac{\xi}{c}} &\text{ if } \xi>0,\\
    1&\text{ if } \xi<0.\end{array}\right.
    \ee
    Then $(N_{HS}, P_{HS})(x-ct)$ is a traveling wave moving at speed $c$ solution of the Hele-Shaw system
    \begin{align}
     c N' + (N P')'+ N \Phi(P)&=0, \label{Heleshaw1}\\
     0 \leq N \leq 1, \quad (1-N)P &= 0, \quad P \geq 0,\label{Heleshaw2}\\
     P(P'' + \Phi(P)) = 0. \label{Heleshaw3}
    \end{align}
    
    \item Let $(N,P)\in L^\infty(\R) \times W^{1,\infty}(\R)$ be a traveling wave profile moving at speed $c$ of the Hele-Shaw system \eqref{eq:HS}. Then there exists $\xi_0\in \R$ such that $(N,P)=(N_{HS}, P_{HS}) (\cdot - \xi_0)$.

\end{enumerate}

\end{lem}

\begin{rmk}
\begin{itemize}
    \item The Lipschitz regularity assumption on $P$ ensures that the term $P'N$ is well-defined, as a product of two $L^\infty$ functions. 
    
    \item An important feature of the analysis is the continuity of the flux $( c +P') N$ on $\R$ (and in particular at the transition point $\xi_0$). This property will determine the value of $N(\xi_0^+).$

\end{itemize}

\label{rmk:flux}
\end{rmk}
\begin{proof}
It is easily checked that $(N_{HS}, P_{HS})$ is a solution of~\eqref{Heleshaw1}-\eqref{Heleshaw3}.
Hence the difficulty is to prove that all solutions are equal to $(N_{HS}, P_{HS})$ (up to a translation). As emphasized in Remark \ref{rmk:flux}, the flux $J= cN+ NP'$ satisfies $J'=-N \Phi(P)\in [-1,0]$. Hence $J$ is Lipschitz continuous and decreasing. Using the values of $N,P$ at $\pm\infty$, we find that $J(-\infty)=c$, $J(+\infty)=0$, and therefore $0\leq J\leq c$ a.e.

Since $P$ is Lipschitz continuous, the set $\{P>0\}$ is a countable union of disjoint open intervals, say $\cup_{j\in \mathcal J} (a_j, b_j)$. On any such interval $(a_j, b_j)$, we have $N=1$ and 
\begin{equation*}\label{HS}
    -P''=1-P , \quad \forall\xi \in (a_j, b_j).
\end{equation*}
Hence there exist $C_j^\pm$ such that
\[
P(\xi)= 1 + C_j^+ e^\xi + C_j^- e^{-\xi}\quad \forall \xi \in (a_j, b_j) .
\]
Note that the case $b_j=+\infty$ is excluded, since $N(+\infty)=0$, and that $C_j^-=0$ if $a_j=-\infty$. Furthermore, on any interval $(a_j, b_j)$, we have $J=c + P'\in [0,c]$, and $J'=P''\leq 0$. Hence $P$ is non-increasing and concave on $(a_j,b_j)$. If $a_j, b_j\in \R$, we have additionally $P(a_j)=P(b_j)=0$, since $a_j, b_j \in \p \{P>0\}$. This entails  that $P(\xi)=0$ for all $\xi\in (a_j, b_j)$, which is absurd.
Hence $\mathcal J$ is a singleton and there exists $\xi_0\in \R$ such that
$\{P>0\}=(-\infty, \xi_0)$. Furthermore, since $P(\xi_0)=0$, we find that
\begin{equation}\label{p_nonnul}
P(\xi)=1-e^{\xi-\xi_0}\quad \forall \xi<\xi_0.
\ee

Let us now consider the free-phase, i.e. 
the set $\{P=0\}= [\xi_0, +\infty[$. In (the interior of) this interval, the equation becomes
\begin{equation*}
    {c}N'=-N,\quad\forall\xi>\xi_0.
\end{equation*}
The solution of the above linear equation is of the form
\begin{equation*}
    N(\xi)=C\exp\left(-\frac{\xi-\xi_0}{\bar{c}}\right).
\end{equation*}
We infer that in $(\xi_0, +\infty)$, $J= c C \exp\left(-\frac{\xi-\xi_0}{\bar{c}}\right)$.
By continuity of $J$ at $\xi=\xi_0$, we obtain
\[
c-1= J(\xi_0^-)=J(\xi_0^+)= cC.
\]
Thus $C=(c-1)/c$, and we find that $(N,P)=(N_{HS}, P_{HS})(\cdot -\xi_0)$.

\end{proof}

\subsection{Qualitative properties of travelling waves for the porous medium equation \eqref{eq:pde-0}}
\label{limgammainf}

Let us now consider traveling waves for the porous medium equation \eqref{eq:pde-0}. We are interested in the behavior of such profiles in the limit $\gamma\to +\infty$, with a fixed velocity $c>0$. 
In the following two subsections, we aim to derive qualitative and quantitative information on the profiles when $\gamma\gg 1$.


The existence of a profile $\Ng$ solution to~\eqref{eq:Ng}-\eqref{eq:Ng-endpt} is ensured by a former study of Gilding and Kersner \cite{gilding2005fisher}.
More precisely, as a particular case of~\cite{gilding2005fisher}, one has the following result.

\begin{thm}[Gilding \& Kersner \cite{gilding2005fisher}]\label{t:gilding} 
Let  $\cg^*= \sqrt{\frac{\gamma}{\gamma+1}}$.
\begin{enumerate}
\item System~\eqref{eq:Ng}-\eqref{eq:Ng-endpt} has a unique solution $N_\gamma$ (up to a shift) for every $c\geq \cg^*$ and no solution for $c< \cg^*$.

    \item When $c= \cg^*$, $N_\gamma$ is a sharp front, i.e. the support of $N_\gamma$ is bounded above, and, modulo translation,
    \begin{equation*}
        N_\gamma(\xi)=\left\{\begin{array}{rcc}
           \left( 1-\exp\left(c\xi\right)\right)^{1/\gamma} &\text{for}&\xi<0,\\
             0& \text{for}&\xi\geq 0.
        \end{array}\right.
    \end{equation*}
    \item When $c>\cg^*$, $N_\gamma$ is positive, strictly monotonic and satisfies
    \begin{equation}\label{cvg-Ng-minf}
    (\ln(1 - N_\gamma))'(\xi)\rightarrow 
    \sqrt{1+\dfrac{c^2}{4\gamma^2}} - \dfrac{c}{2\gamma}= \frac{1}{ \sqrt{1+\dfrac{c^2}{4\gamma^2}} + \dfrac{c}{2\gamma}} \quad\text{as}\quad\xi\rightarrow-\infty,
    \end{equation}
    and
    \begin{equation*}
        (\ln(N_\gamma))'(\xi)\rightarrow-\frac{1}{c},\quad\text{as}\quad\xi\rightarrow+\infty.
    \end{equation*}
\end{enumerate}
\end{thm}

The above theorem guarantees the existence (and the uniqueness up to a shift) of a TW $\Ng$ for all $c \geq \cg^* = \sqrt{\dfrac{\gamma}{\gamma+1}}$; the smoothness of $N_\gamma$ when $c>\cg^*$; the monotonically decreasing behavior of $\Ng$ and its boundness on $\R$. Notice that the sharp front with minimal speed $c = \cg^*$ it is only Hölder continuous with exponent $1/\gamma$ at $\xi = 0$. The fact of $\Ng^{\gamma+1}$ being continuously differentiable in the whole domain means this traveling wave is a weak solution in the usual sense, while from the physics perspective, it indicates the presence of continuous flux. 

Theorem \ref{t:gilding} is adapted from Theorem 1 in \cite{gilding2005fisher}, and therefore, we refer to this work for detailed proof.

\begin{figure}[ht]
\begin{center}
\includegraphics[scale=0.35]{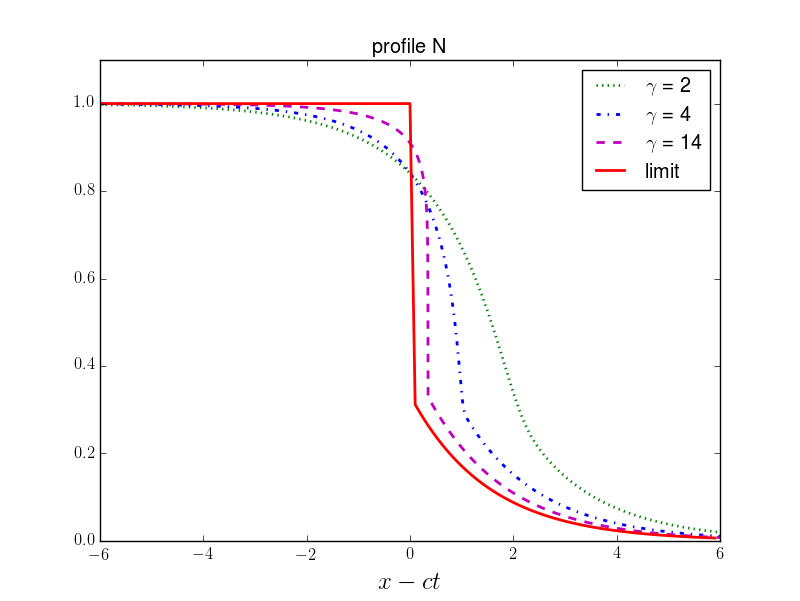}\includegraphics[scale=0.35]{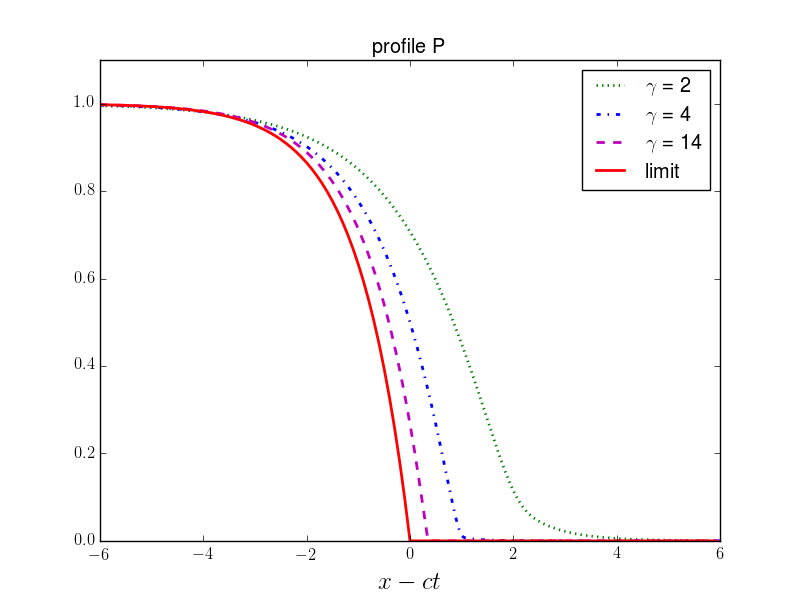}
\end{center}
\caption{{\label{fig:profiles}Density and pressure profiles for finite values of $\gamma$ and limit profiles, $c=1.5$.} 
}
\end{figure}
From now on, we  pick a velocity $c>1$ independent of $\gamma$, so that $c>\cg^*$\footnote{All the results of this paper remain true with little or no modification if the velocity $\cg$ depends on $\gamma$ in such a way that $\cg \to \bar{c}$ with $\bar{c} > 1$. However for the sake of readability we have chosen $\cg\equiv c>1$.}.
We also  fix the shift in $N_\gamma$ by imposing
\begin{equation}\label{eq:norm-N0g}
\Ng(0) = \left(\dfrac{1}{\gamma}\right)^{\frac{1}{\gamma}}.
\end{equation} 

\medskip

The goal of this subsection is to prove the following result:

\begin{prop}
Let $c>1$ and let $(N_\gamma,\Pg)$, $\Pg:=\pg(\Ng)$, be the unique bounded weak solution to \eqref{eq:Ng} satisfying \eqref{eq:norm-N0g}. 
Let $(N_{HS}, P_{HS})\in L^\infty\cap W^{1,\infty}(\R)$ be the reference traveling wave solution moving with speed $c$ of the Hele-Shaw system, see \eqref{def:barN-P}.

\begin{enumerate}
    \item The following convergence properties hold:
    \begin{itemize}
    \item \textit{Weak-star convergence:}\begin{align*}
\Ng & \rightharpoonup N_{HS}\quad \text{in } {w^{*}-L^{\infty}},\quad\Pg  \rightharpoonup P_{HS} \quad \text{in } {w^{*}-W^{1,\infty}};\\
\end{align*} 
\item for any compact set $K\subset\R$
\begin{align*}
\Pg &\to P_{HS} \quad \text{in } {C(K)};
\end{align*}
\item $\Ng \to 1 $ uniformly on $\R_-$ and $\Pg'\to P_{HS}'$ uniformly in $\mathcal C(]-\infty, 0])$. 
\end{itemize}

\item Pointwise bounds for $\Pg$ on $\R_-$: setting $\lambda=(-c+ \sqrt{c^2 +4})/2$, we have, 
\begin{equation}\label{eq:encadr_P_R-}
1- \left( 1 - \frac1\gamma\right) e^{\lambda \xi}\leq \Pg \leq 1 - \left( 1 - \frac1\gamma\right) e^{\xi}, \quad \forall \ \xi \leq 0;
\end{equation}


\end{enumerate}

\label{prop:qualitative-Pg-Ng}
\end{prop} 
The rest of this subsection is devoted to the proof of Proposition \ref{prop:qualitative-Pg-Ng}.

\noindent	\textbf{$L^\infty$ bounds.} From the maximum packing constraint, we know that $0 \leq P_\gamma \leq p_M=1$. Then, standard comparison provides 
\[
\ba
0\leq N_\gamma \leq \underset{\gamma\rightarrow + \infty}{\lim}\left(p_M\right)^{1/\gamma}=1.\ea
\]
From Theorem \ref{t:gilding}, we know that $ N_\gamma' \leq 0$, which combined with the definition of the stiff pressure yields  $P_\gamma' \leq 0$.

Therefore there exist $(N,P)\in L^\infty\times L^\infty(\R)$ such that up to the extraction of a subsequence, $N_\gamma \rightharpoonup N$, $P_\gamma \rightharpoonup P$ in $w^*-L^\infty(\R)$. Furthermore, $N, P$ are non-increasing. The choice of shift \eqref{eq:norm-N0g} implies that $\Ng(0)\rightarrow 1, \Pg(0)\to 0$. 
Hence, since $\Ng$ is non-increasing $\Ng$ converges uniformly towards 1 on $]-\infty, 0]$, and $\Pg$ converges uniformly towards zero on $[0, +\infty[$. It follows that
$N(\xi)=1$ for $\xi<0 $ and $P(\xi)=0$ for $\xi>0$.

\medskip
\noindent\textbf{Strong convergence  of $\Pg$ and $J_\gamma$.} Define the flux $J_\gamma:= c \Ng + \gamma \Ng' \Ng^\gamma= c \Ng + \Ng \Pg'$. We observe that equation \eqref{eq:Ng} can be written as
\[
J_\gamma '= -\Ng (1-\Ng^\gamma),
\]
so that $J_\gamma$ is decreasing on $\R$. 
Combining the latter with the $L^{\infty}$ bounds on $\Ng$ yields
\begin{equation}\label{est:boundJ'}
\begin{split}
-1 \leq J_{\gamma}'&\leq 0,\\
0=J_\gamma(+\infty)\leq J_\gamma&\leq J_\gamma (-\infty)= c.
\end{split}
\end{equation}
In particular, $c\Ng+ \Ng \Pg'\geq 0$. Since we already know that $\Pg$ is non-increasing, it follows that 
\be\label{est:Pg-lipschitz}
-c\leq \Pg'\leq 0,\quad 0\leq \Pg\leq 1.
\ee

From inequality \eqref{est:boundJ'} (resp. \eqref{est:Pg-lipschitz}) and Ascoli's theorem, $J_\gamma$ (resp. $\Pg$) converges strongly, up to a subsequence, in $\mathcal C(K)$ for any compact set $K\subset \R$.  Note also that $\Pg' \overset{*}{\rightharpoonup} P'$ in $L^\infty(\R)$; since  $\Ng$ converges uniformly towards 1 on $\R_-$, we find that $J=cN+NP'$ on $]-\infty, 0[$. 

The exact same cannot be done with $\Ng$. Indeed, from \eqref{eq:Ng} and \eqref{est:boundJ'}, we can deduce the following bounds for $\Ng$ 
\begin{equation}\label{est:boundNg'}
    -c\frac{\Ng}{\Ng^{\gamma}\gamma} \leq  \Ng'\leq \frac{c(1-\Ng)}{\Ng^{\gamma}\gamma}.
\end{equation}
Note that obtaining an $L^{\infty}$ bound implies controlling $\Ng^{1-\gamma}\gamma^{-1}$ in $L^{\infty}$ over any compact on $\mathbb{R}$ when $\gamma\rightarrow+\infty$. This is impossible as $\Ng\in(0,1)$. In fact, we show in what follows that $N$ is discontinuous in $\xi=0$.

\medskip
\noindent	\textbf{Passing to the limit in equation \eqref{eq:Ng}.} We can write the diffusion term as
\[
(\Ng \Pg')'=(\gamma \Ng' \Ng^\gamma)'=\frac{\gamma}{\gamma+1}(\Ng^{\gamma+1})''=\frac{\gamma}{\gamma+1} (\Pg \Ng)''.
\]
Since $\Pg$ converges strongly in $\mathcal C(K)$ for all compact set $K\subset \R$, while $\Ng$ converges weakly-* in $L^\infty(\R)$, we can pass to the (weak) limit in equation \eqref{eq:Ng}.

We obtain that $(N,P)$ satisfies the following equation in the sense of distributions
\be\label{eq:limit}
-cN' - (NP)''=N(1-P).
\ee
The same argument also shows that $J=cN + (NP)'$ on $\R$.

\medskip
\noindent
\textbf{Limit in the free-phase ($\xi>0$).}
We recall that $P=0$ in $\R_+$. Hence, in $(0, + \infty)$, equation \eqref{eq:limit} becomes 
\[
-cN'=N.
\]
We recognize the ODE satisfied by $N_{HS}$ in the free-phase in the Hele-Shaw system.
It follows that
\[
N(\xi)=C\exp \left(-\frac{\xi}{c}\right)\quad \forall \xi>0,
\]
for some $C>0$.

\medskip
\noindent
\textbf{Limit in the congested phase $(\xi<0)$.} We recall that $N=1$ on $]-\infty, 0[$. Inserting this information into \eqref{eq:limit},
 the following elliptic equation (complementarity equation) is obtained
	\begin{equation}\label{eq:compl-p}
	P'' + (1-P) = 0,\quad\text{in}\quad\mathcal{D}'((-\infty,0)).
	\end{equation}
	From
	$P(0) = 0$ (recall that $P$ is continuous), it follows that $P(\xi) = 1- e^{\xi}$ for $\xi \in \R_-$.

\medskip
\noindent
\textbf{(N,P) satisfies \eqref{Heleshaw2}.} We know that $P=0$ on $[0,+\infty)$ and $N=1$ on $\mathbb{R}_{-}$; hence, $P(1-N)=0$ on $\mathbb{R}$ as in \eqref{Heleshaw2}.

\medskip
\noindent
\textbf{Jump relation at $\xi=0$.} 
We recall that the flux $J=cN + (NP)'$ is continuous on $\R$, and in particular at $\xi=0$.
 Thus,
\begin{equation}\label{est:jump}
\lim_{\xi \to 0^+} N(\xi) = 1 - \frac{1}{c}.
\end{equation}
Gathering all the information, we find that $(N,P)=( N_{HS}, P_{HS})$. Furthermore, since the limit is uniquely identified, we deduce that the whole sequence $(\Ng, \Pg)$ converges (in the sense given above).

\medskip
\noindent
\textbf{Sub- and super-solution for $\Pg$ on $\R_-$.}
 Using \eqref{est:Pg-lipschitz}, it follows that
\[
-\Pg'' \Ng = \Ng (1-\Pg) + (c+ \Pg') \Ng' \leq \Ng (1-\Pg),
\]
whence
\[
-\Pg'' \leq 1-\Pg\quad \text{on }\R.
\]
Now, let $\xi_1\in \R$ be arbitrary, and let $P_1:=\Pg(\xi_1)$.  We have for $P_+:= 1- (1-P_1) e^{\xi-\xi_1}$ that $- P_+''=1- P_+$. Furthermore, 
\[
-(\Pg- P_+)''\leq 1-(\Pg- P_+) \text{ on }(-\infty,\xi^{*}).
\]
It follows from the maximum principle that $P_\gamma\leq  P_+$ for $\xi\leq \xi^*$. In particular, taking $\xi_1=0$ and $P_1=1/\gamma$,
\begin{equation}\label{est:superPg}
0\leq \Pg(\xi) \leq 1 - \left(1-\frac{1}{\gamma}\right) e^\xi\quad \forall \xi \leq 0.
\end{equation}
In a similar fashion, recalling that $\gamma \Pg\geq 1$ on $\R_-$ and $\Pg'\leq 0$, we have
\[
-\Pg''= 1-\Pg + \frac{c\Pg'}{\gamma\Pg} + \frac{(\Pg')^2}{\gamma \Pg}\geq 1-\Pg + c\Pg'.
\]
Arguing as before, we define $P_-(\xi)= 1 - \left(1-\frac{1}{\gamma}\right) e^{\lambda \xi}$, where $\lambda$ is the positive root of $\lambda^2 + c\lambda-1=0$ (i.e. $\lambda=(-c + \sqrt{c^2 +4})/2$). By definition of $\lambda$, $P_-$ satisfies
\[
-P_-''= 1-P_- + cP_-', \quad \lim_{\xi\to -\infty} P_-(\xi)=1,\ P_-(0)=\frac{1}{\gamma}.
\]
We infer that 
\begin{equation}\label{est:subPg}
1 - \left(1-\frac{1}{\gamma}\right) e^{\lambda\xi} \leq \Pg(\xi) \qquad \forall \ \xi \leq 0.
\end{equation}

\medskip
\noindent
\textbf{Uniform convergence of the flux and of $\Pg'$ on $\R_-$.}

We recall that $J_\gamma'= -\Ng (1-\Pg)$. The pointwise bounds on $\Pg$ imply that
\[
|J_\gamma'|\leq e^{\lambda \xi}\quad \forall \xi\leq 0,\ \forall \gamma>0.
\]
It follows immediately that $J_\gamma$ converges towards $J_{HS}$ uniformly in $\mathcal C(\R_-)$. Since
\[
\Pg'= \frac{J_\gamma}{\Ng}- c,
\]
we infer that $\Pg'$ also converges uniformly towards $ P_{HS}'$ in $\mathcal C(\R_-)$.
This concludes the proof of Proposition \ref{prop:qualitative-Pg-Ng}.\qed

\subsection{Phase portrait of $\Ng$ and further consequences}
\label{ssec:phase-portrait}

In this subsection, we derive other properties of the family $(N_\gamma)_{\gamma>0}$, which will be useful in our stability analysis. 
These properties rely crucially on the analysis of the phase portrait of $\Ng$. 

In order to plot the phase portrait of $\Ng$, we use the results of \cite{medvedev2003travelling}, together with the following remark: using equation \eqref{eq:Ng}, we have
\begin{eqnarray*}
\frac{d \Ng '}{d\Ng}&=& \frac{d \Ng '}{d\xi}\frac{d\xi}{d\Ng}\\
&=& -\frac{1}{\gamma \Ng^\gamma \Ng'}\left[ c \Ng' + \gamma^2 (\Ng')^2 \Ng^{\gamma-1} + \Ng (1-\Ng^\gamma)\right].
\end{eqnarray*}
Hence $d\Ng'/d\Ng$ vanishes if and only if $\Ng^\gamma(1-\Ng^\gamma)\leq c^2/(4\gamma^2)$ and  $\Ng'\in \{Q_-(\Ng), Q_+(\Ng)\}$, where
\be\label{def:Q_pm}
Q_\pm(\Ng)= \frac{1}{2 \gamma^2 \Ng^{\gamma-1}}\left(- c \pm \sqrt{c^2 -4\gamma^2 \Ng^\gamma(1-\Ng^\gamma)}\right).
\ee
Note that the curves $\Gamma_\pm=\{(N,Q_\pm(N)),\ N\in (0,1)\}$ each consist of two branches, for $N\in (0, N_1)$ and $N\in (N_2,1)$.
 The points $N_i$ are the roots of the discriminant, i.e. $N_i^\gamma (1-N_i^\gamma)= c^2/(4\gamma^2)$. The curves $\Gamma_+$ and $\Gamma_-$ intersect at $N=N_1$ and at $N=N_2$.
 A straightforward analysis shows that
\[
N_1=1-\frac{2\ln \gamma}{\gamma} + o\left(\frac{\ln \gamma}{\gamma} \right) ,\quad N_2= 1-\frac{c^2}{4 \gamma^3} + o \left( \frac{1}{\gamma^3}\right),
\]
with 
\[
Q_\pm (N_1)\sim - \frac{2}{c},\quad Q_\pm (N_2) \sim - \frac{c}{2\gamma^2}.
\]
Furthermore, $Q_+(N)\sim -\frac{N}{c}$ for $N\ll 1$, while $Q_-(N)\to -\infty$ as $N\to 0$, and \vbox{$Q_+(N)\sim -\gamma(1-N)/c$} for $1-N\ll 1$, while $Q_-(1)=-c/\gamma^2$.

Note also that with the normalisation of the previous section, i.e. $\Ng(0)=\gamma^{-1/\gamma}$, we have $\Ng(0)\in [N_1, N_2]$.

Now, let us denote by $\mathcal T$ (resp. $\mathcal S$) the interior region between the curves $\Gamma_-$ and $\Gamma_+$ for $0<N<N_1$ (resp. $N_2<N<1$). 
We also denote by $\Gamma$ the curve $(\Ng, \Ng')$, which we orientate in the direction of growing $\Ng$.
We make the following observations:
\begin{enumerate}[label=(\roman*)]
    \item For all $\Ng\in (N_1, N_2)$, $d\Ng'/d \Ng \geq 0$;
    
    \item For all $\Ng \in (0, N_1)$ (resp. \vbox{$\Ng\in (N_2, 1)$}), $d\Ng'/d \Ng < 0$ iff $(\Ng, \Ng')\in \mathcal T$ (resp. $(\Ng, \Ng')\in \mathcal S$);
    
    \item If $\Gamma$ crosses one of the curves $\Gamma_\pm$, then $d\Ng'/d \Ng=0$ at the crossing point and therefore the tangent to $\Gamma$ at the crossing point is horizontal;
    
    \item $\displaystyle \frac{d Q_\pm}{dN}\gtrless 0$ for all $N\in (N_2, 1)$;
    
    \item $\displaystyle \frac{d Q_\pm}{dN}\lessgtr 0$ for all $N\in (0,N_1)$;
    
    \item When $\xi\to -\infty$, we have $\Ng(\xi)\to 1$, and $\Ng'(\xi)\sim -\left(\sqrt{1 + \dfrac{c^2}{4 \gamma^2}} - \dfrac{c}{2\gamma}\right)(1-\Ng(\xi))$.
\end{enumerate}

\begin{figure}[ht]
\begin{center}
\includegraphics[scale=0.45]{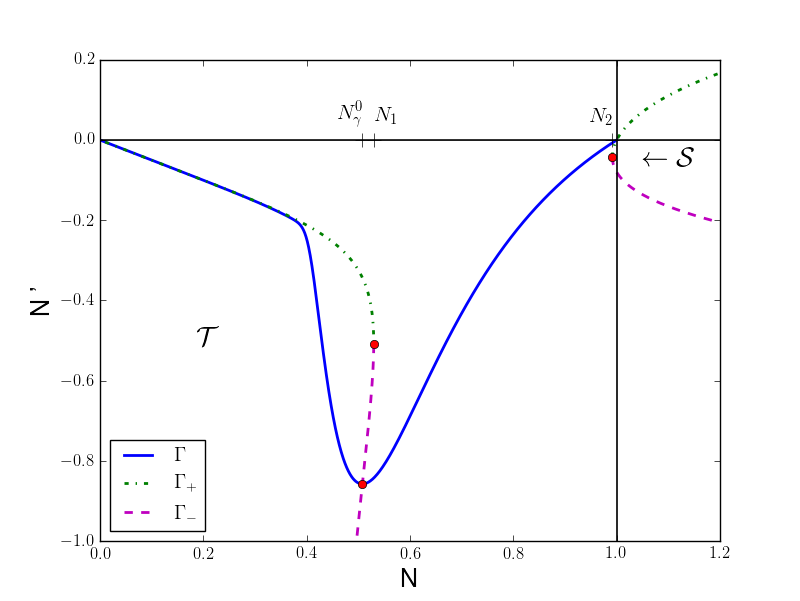}
\end{center}
\caption{{\label{fig:phaseplane}Trajectory $\Gamma$ in the phase plane $(N,N')$, $c=2$, $\gamma=5$.} 
}
\end{figure}
The proof of all items is easy and left to the reader, except for (v), which we prove below.
Note that (vi) is a consequence of \eqref{cvg-Ng-minf}. 
It follows from (vi) that for $\Ng$ in a neighborhood of $1$ (the size of which depends on $\gamma$), the curve $\Gamma$ is above $\Gamma_+$. Furthermore, (iii) and (iv) imply that if the curve $\Gamma$ intersects the region $\mathcal S$, then it cannot exit $\mathcal S$. It follows that $\Gamma$ lies strictly above $\Gamma_+$ for all $N\in (N_2, 0)$. 
\emph{Consequently, for all $\Ng \in (N_1, 1)$, $d\Ng'/d\Ng\geq 0$. }

Let us now prove that $dQ_+/dN\leq 0$ for all $N\in (0, N_1)$ (the inequality for $Q_-$ is easier and left to the reader). We have, setting $P=N^\gamma$, 
\begin{eqnarray*}
\frac{dQ+}{dN}&=&\frac{d}{dN}\left[ - \frac{c}{2\gamma^2 N^{\gamma-1}} \left( 1-\sqrt{1 - \frac{4\gamma^2}{c^2} N^\gamma (1-N^\gamma)}\right)\right]\\
&=& - \frac{c}{2\gamma^2 N^{\gamma-1}}\left[-\frac{\gamma-1}{N}\left( 1-\sqrt{1 - \frac{4\gamma^2}{c^2} N^\gamma (1-N^\gamma)}\right)
+ \frac{2\gamma^3}{c^2} \frac{N^{\gamma-1} (1-2 N^\gamma)}{\sqrt{1 - \frac{4\gamma^2}{c^2} N^\gamma (1-N^\gamma)}}\right]\\
&=& - \frac{c}{2\gamma^2 N^\gamma}\left[ -\frac{4 \gamma^2(\gamma-1)P (1-P)}{c^2(1 + \sqrt{ 1- \frac{4\gamma^2}{c^2} P(1-P)})} + \frac{2\gamma^3 P(1-2P)}{c^2\sqrt{ 1- \frac{4\gamma^2}{c^2} P(1-P)} }\right]\\
&=&- \frac{\gamma (1-2P) - (\gamma-2 + 2P) \sqrt{ 1- \frac{4\gamma^2}{c^2} P(1-P)}}{c(1 + \sqrt{ 1- \frac{4\gamma^2}{c^2} P(1-P)})\sqrt{ 1- \frac{4\gamma^2}{c^2} P(1-P)} }.
\end{eqnarray*}
Note that for $0<N<N_1$, $P=O(1/\gamma^2)$. In this regime, it can be easily checked that the numerator of the right-hand side is positive, and therefore $dQ_+/dN < 0$ for all $N\in (0,N_1)$.
This completes the proof of (v).

We deduce that for $N\in (0, N_1)$, the curve $\Gamma$ can cross $\Gamma_+$ at most once, as $\Gamma$ exits the region $\mathcal T$. This completes the proof of (v).
Now, let us argue by contradiction and assume that there exists $\Ng^0\in (0, N_1)$ such that $(\Ng^0, (\Ng^0)')\in \Gamma$ lies above $\Gamma_+.$
Then there are two possibilities:
\begin{itemize}
    \item either $(\Ng, \Ng')$ is above $\Gamma_+$ for all $\Ng\in (0, N_1)$. In that case, $d\Ng'/d\Ng\geq 0$ for all $\Ng\in (0,N_1)$. Since $(0,0)\in \Gamma$ and $\Ng'\leq 0$, it follows that $\Ng'=0$ for all $\Ng\in (0, N_1)$, which contradicts Theorem \ref{t:gilding}.
    
    \item or there exists $\Ng^1\in (0, \Ng^0)$ such that $(\Ng^1, (\Ng^1)')\in \Gamma\cap \mathcal T$. In that case, since $\Gamma$ and $\Gamma_+$ intersect at most once, there exists $N_3\in (0, N_1)$ such that for all $\Ng \in (0, N_3)$, $(\Ng, \Ng')\in \Gamma \cap \mathcal T$ and for $\Ng\in  (N_3, N_1)$, $(\Ng, \Ng')$ is above $\Gamma_+$. Hence $\Ng'$ reaches a minimum for $\Ng=N_3$, and the value of this minimum is $Q_+(N_3)\geq Q_+(N_1)\sim-2/c$. Thus $\Ng'$ is bounded in $L^\infty$. Using Ascoli's theorem, we infer that $\Ng$ converges uniformly on  $\mathcal C(K)$ for any compact set $K\subset \R$ as $\gamma\to \infty$. Since $N_{HS}$ is discontinuous at $\xi=0$, we have reached a contradiction. 
\end{itemize}
We conclude that $(\Ng, \Ng')$ remains below $\Gamma_+$ for all $\Ng\in (0, N_1)$, and therefore $\Gamma$ does not cross $\Gamma_+$. 
Using once again the fact that $\min \Ng'$ must blow up as $\gamma\to \infty$, we infer that $\Gamma$ and $\Gamma_-$ intersect exactly once, at a point where $\Ng=\Ng^0\in (0, N_1)$, and $\Ng^0$ is such that $Q_-(\Ng^0)\to -\infty$ as $\gamma\to +\infty$. For all $\Ng\in (0, \Ng^0)$, $d\Ng'/d\Ng\leq 0$, and for $\Ng\in (\Ng^0, 1)$, $d\Ng'/d\Ng\geq 0$. Thus we obtain the phase portrait drawn in Figure \ref{fig:phaseplane}.

Let us now go back to the analysis of $\xi\in \R \mapsto \Ng(\xi)$. There exists a unique $\xi^0_\gamma\in \R$ such that $\Ng(\xi^0_\gamma)=\Ng^0$. Note that $d\Ng'/d\Ng$ and $\Ng''$ have opposite signs. Hence, $\Ng$ is concave on $(-\infty, \xi^0_\gamma)$ and convex on $(\xi^0_\gamma, + \infty)$. 
We are now ready to prove the following Lemma:
\begin{lem}
\label{lem:transition}
We normalize the function $\Ng$ so that $\Ng(0)= \gamma^{-1/\gamma}$. We have the following properties:
\begin{itemize}
    \item $\xi^0_\gamma>0$ and $\lim_{\gamma\to \infty} \xi^0_\gamma=0$;
    
    \item $\sup_{\gamma>0} \sup_{\xi <0} |\Ng'(\xi)| <+\infty$ and $\|\Ng'\|_{L^\infty(\R)}=-Q_-(\Ng^0)\to + \infty$ as $\gamma \to \infty$;
    
    \item $\lim_{\gamma\to \infty } \Ng^0= 1 - c^{-1}$; 
    
    \item For $\gamma $ large enough, for  all $\xi\geq \xi^0_\gamma$,
    \[
    0\leq \Ng(\xi)\leq \Ng^0 \exp\left( - \frac{1}{2c} (\xi-\xi_0)\right);
    \]
    
    \item $\Pg'\to P_{HS}'$ and $\Ng\to N_{HS}$ in $L^p_\text{loc} (\R)$ for all $p\in [1, + \infty[;$
    
    \item Let $\xi^*_\gamma>\xi^{0}_\gamma$ such that $\Ng'(\xi^*_\gamma)=-\frac{1}{c}\left(1-\frac{1}{c}\right)$. Then $\xi^*_\gamma \to 0$  and $\Ng(\xi^*)\to 1 - c^{-1}$ as $\gamma \to \infty$.

\end{itemize}

\end{lem}
\begin{proof}
{\it First step: Upper-bound on $\Ng^0$ and on $\xi^0_\gamma$.}

The analysis of the phase portrait entails immediately that $\|\Ng'\|_{L^\infty(\R)}=-Q_-(\Ng^0)$. As recalled above, if $Q_-(\Ng^0)$ remains bounded, then $\Ng$ converges strongly in $\mathcal C(K)$ for any compact set $K$, which is absurd since $N_{HS}$ is discontinuous. Hence $Q_-(\Ng^0)$ must blow up. Since
\[
-\frac{c}{\gamma^2 (\Ng^0)^{\gamma-1}}\leq Q_-(\Ng^0) \leq -\frac{c}{2\gamma^2 (\Ng^0)^{\gamma-1}},
\]
we deduce that $(\Ng^0)^\gamma= o(\gamma^{-2}))\ll \gamma^{-1}=\Ng(0)^\gamma$. Thus $\xi^0_\gamma>0$. 

Since the flux $J_\gamma$ is decreasing on $\R$, it follows that $J_\gamma(\xi^0_\gamma)\leq J_\gamma (0)$. Now
\be\label{eq:NgO-J}
J_\gamma(\xi^0_\gamma)= c \Ng^0 + \gamma (\Ng^0)^\gamma Q_-(\Ng^0) = \left( c + O\left(\frac{1}{\gamma}\right)\right) \Ng^0,
\ee
and
\[
J_\gamma(0)=\gamma^{-1/\gamma}\left( c + \Pg'(0)\right).
\]
Thanks to the sub- and super-solutions for $\Pg$ on $\R_-$, we know that for all $\xi<0,$
\[
\left(1 - \frac{1}{\gamma}\right) (1-e^{\lambda\xi})\leq \Pg - \Pg(0)\leq \left(1 - \frac{1}{\gamma}\right) (1-e^{\xi}),
\]
where $\lambda=(\sqrt{c^2+4} - c)/2$. Hence $\Pg'(0) \in [-(1-\gamma^{-1}), (1-\gamma^{-1})\lambda]$. We deduce that $J(0)\leq c - (1-\gamma^{-1})\lambda$, and therefore
\be\label{bound-Ng0}
\Ng^0 \leq \left( c + O\left(\frac{1}{\gamma}\right)\right)^{-1} \left( c - \left(1-\frac{1}{\gamma}\right)\lambda\right)\leq 1-\frac{\lambda}{c} + O\left(\frac{1}{\gamma}\right).
\ee
Hence $\limsup_{\gamma\to \infty} \Ng^0 \leq (c-\lambda)/c<1$.

The bound on $\Pg'(0)$ also implies the boundedness of $\Ng'$ on $\R_-$. Indeed, since $\xi^0_\gamma>0$, $\Ng'$ is decreasing and negative on $\R_-$, and
\[
\sup_{\xi<0}|\Ng'(\xi)|= |\Ng'(0)| = -\frac{\Pg'(0) }{\gamma \Ng(0)^{\gamma-1}}= O(1).
\]

Let us now address the upper-bound on $\xi^0_\gamma$. We recall that $\Ng$ is concave on $(-\infty, \xi^0_\gamma)$. Consequently, for all $\xi \in (0, \xi^0_\gamma)$, 
\[
0\leq \Ng(\xi) \leq \Ng(0) + \Ng'(0) \xi.
\]
In particular, taking $\xi=\xi^0_\gamma$, we deduce that
\[
\xi^0_\gamma \leq -\frac{\Ng(0)}{\Ng'(0)}= - \frac{\gamma \Pg(0)}{\Pg'(0)}= - \frac{1}{\Pg'(0)},
\]
since $\Pg(0)= \gamma^{-1}$ by choice of our normalization.
We deduce in particular that 
\[
0\leq \xi^0_\gamma \leq \frac{1}{\lambda(1-\gamma^{-1})}.
\]

\medskip

\noindent{\it Second step: Super-solution for $\Ng$ on $(\xi^0_\gamma, +\infty)$.}

We recall that $\Ng$ is convex on $(\xi^0_\gamma, +\infty)$. As a consequence, using the equation on $\Ng$, we have
\begin{equation}\label{conv-free-zone}
    -c \Ng'= \Ng (1-\Ng^\gamma) + \gamma \Ng'' \Ng^{\gamma} + \gamma^2 (\Ng')^2 \Ng^{\gamma-1} \geq \Ng (1-(\Ng^0)^\gamma)\quad \forall \xi \in (\xi^0_\gamma, +\infty).
\end{equation}
The Gronwall Lemma then implies that
\be\label{in:pointwise-Ng}
\Ng(\xi) \leq \Ng^0 \exp \left( - \frac{1-(\Ng^0)^\gamma}{c}(\xi-\xi^0_\gamma)\right)\quad \forall \xi\geq \xi^0_\gamma.
\ee
Recalling \eqref{bound-Ng0}, we deduce that for $\gamma$ large enough, for all $\xi\in (\xi^0_\gamma, +\infty)$,
\be\label{super-sol-Ng}
\Ng(\xi) \leq \Ng^0 \exp \left( - \frac{1}{2c}(\xi-\xi^0_\gamma)\right).
\ee

\noindent{\it Third step: Strong convergence of $\Pg'$ and $\Ng$.}

We start with an $L^2$ bound for $\Pg'$.
From \eqref{eq:Ng}, $\Pg$ is solution to
\begin{equation}\label{eq:Pg}
-\cg \Pg' - \gamma \Pg \Pg'' - (\Pg')^2 = \gamma \Pg(1-\Pg).
\end{equation}
Integrating Equation~\eqref{eq:Pg} over $\R$ gives
\[
(\gamma-1) \int_{\R} |\Pg'(\xi)|^2 d\xi  = -\cg + \gamma \int_{\R} \Pg(1-\Pg) \ d \xi.
\]
Hence, we get the following inequality
\begin{align*}
\|\Pg'\|_{L^2(\R)}^2 \leq \dfrac{\gamma}{\gamma-1}\left(\|\Pg\|_{L^1(\R_+)} + \|1-\Pg\|_{L^1(\R_-)} \right).
\end{align*}
The right-hand side is uniformly bounded with respect to $\gamma$ thanks to sub-solution for $\Pg$ on $\R_-$ (see Proposition \ref{prop:qualitative-Pg-Ng}) and to the upper-bound for $\Ng$ on $(\xi^0_\gamma, + \infty)$ (see \eqref{super-sol-Ng}). On the interval $(0,\xi^0_\gamma) $, we simply use the fact that $\xi^0_\gamma$ is bounded and $\Pg \leq \Pg(0)$.
Hence, $(\Pg')_{\gamma>1}$ is bounded in $L^2(\R)$.\\

We now show an additional strong convergence of $(\Pg')_\gamma$ in $L^2$. Going back to Equation~\eqref{eq:Pg} and taking into account that $(\Pg')_{\gamma}$ is uniformly bounded in $L^{2}(\mathbb{R})$, we have for any $\psi \in \mathcal{C}^\infty_c(\R)$:
	\[
	\int_{\R} \psi \Pg \big[\Pg'' + (1-\Pg) \big]\ d \xi
	= -\dfrac{1}{\gamma} \int_{\R} \psi\big[\cg \Pg' + (\Pg')^2\big]\ d \xi
	\to 0 \quad \text{as}~ \gamma \to +\infty.
	\] 
	Hence, by integration by parts in the left-hand side:
	\[
	 \int_{\R} \psi (\Pg')^2 \ d\xi + \int_{\R} \psi' \Pg \Pg' \ d\xi + \int_{\R}\psi \Pg (1-\Pg)\ d\xi \to 0 \quad \text{as}~ \gamma \to +\infty.
	\]
	From the previous bounds, it is clear that
	\[
	\int_{\R} \psi' \Pg \Pg' \ d\xi + \int_{\R}\psi \Pg (1-\Pg) \underset{\gamma \to +\infty}{\to} 
	\int_{\R} \psi' P_{HS} \ P_{HS}' \ d\xi + \int_{\R}\psi P_{HS} (1-P_{HS}) \ d \xi
	= - \int_{\R} \psi (P')^2,
	\]
	using the complementary equation~\eqref{eq:compl-p}. 
	Finally
	\[
	\int_{\R} \psi (\Pg')^2 \ d\xi \to  \int_{\R} \psi (P_{HS}')^2 \ d\xi \quad \text{as}~ \gamma \to +\infty,
	\]
	which means that $(\Pg')_\gamma$ converges strongly in $L^2_\text{loc}(\R)$ to $P_{HS}'$.

We then recall that $J_\gamma= \Ng(c + \Pg')$. Since $(\Pg')_\gamma$ converges in $L^2_\text{loc}$, there exists a sub-sequence (which we still denote by $\Pg'$) which also converges almost everywhere. Recall that $(J_\gamma)_\gamma$ converges in $\mathcal C(K)$ for any compact set $K\subset \R$. 
Therefore $(\Ng)_\gamma$ converges almost everywhere - up to a subsequence - on any set of the form $\cap_{\gamma>0} \{c + \Pg'\geq \delta\}$, with $\delta>0$.

Let $K$ be a compact set in $\R$, and let $M=\sup_K J_{HS}< c$, $m=\inf_K J_{HS}>0$. There exists $\gamma_K>0$ such that for $\gamma \geq \gamma_K$, $J_\gamma\in [m/2, (c+M)/2]$. Since $J_\gamma\leq c+ \Pg'$, we deduce that $c + \Pg'\geq m/2$ on $K$ for $\gamma\geq \gamma_K$. Whence $(\Ng)_\gamma$ converges almost everywhere on $K$, up to a subsequence. Since $\Ng$ is bounded in $L^\infty$, Lebesgue's dominated convergence theorem implies that $(\Ng)_\gamma$ converges towards $N_{HS}$ in $L^p(K)$ for any $p\in [1, +\infty[$. Note that the limit is uniquely identified. Hence the whole sequence $(\Ng)_\gamma$ converges in $L^p_\text{loc}$. 

\noindent{\it Fourth step: Convergence of $\xi^0_\gamma$ and $\Ng^0$.}

We argue by contradiction and assume that $\limsup_{\gamma\to \infty} \xi_\gamma^0>0$. Then there exists $\bar \xi>0$ such that $\bar \xi\leq \xi_\gamma^0$ for a subsequence.
We recall that $\Ng''\leq 0$ on $(0, \bar \xi)$. Passing to the limit in the sense of distributions along this subsequence, we obtain that $N_{HS}''\leq 0$ in $\mathcal D'((0, \bar \xi))$, which is absurd. Thus $\xi_\gamma^0\to 0$ as $\gamma\to \infty$.

Let us now go back to \eqref{eq:NgO-J}. We recall that $J_\gamma$ converges uniformly towards $J_{HS}$, and that $J_{HS}(0)=c-1$. It follows that  $\lim_{\gamma\to \infty} \Ng^0= 1- c^{-1}$.


\noindent{\it Fifth step: Asymptotic behavior of $\xi^*_\gamma$ and $\Ng(\xi^*_\gamma)$.}

First, notice that $\xi^*_\gamma $ is well-defined since $\Ng'$ is  increasing on $(\xi^0_\gamma, +\infty)$, with $\Ng'(+\infty)=0$ and $\lim_{\gamma\to \infty}\Ng'(\xi^0_\gamma)= -\infty$. Furthermore, since $\Ng(\xi^*_\gamma)\leq \Ng^0$, $\limsup_{\gamma\to \infty}\Ng(\xi^*_\gamma) \leq 1- c^{-1}$.

In order to prove that $\lim_{\gamma\to \infty} \xi^*_\gamma=0$, we argue once again by contradiction and we assume that $\limsup_{\gamma \to \infty} \xi^*_\gamma>0$. Thus there exists $\delta>0$ such that $\xi^*_\gamma\geq \delta $ along a subsequence.
By monotony of $\Ng'$, we know that $\Ng'(\xi)\leq \Ng'(\xi^*_\gamma)= -c^{-1} (1- c^{-1})$ for all $\xi \in (\xi^0_\gamma, \delta) \subset (\xi^0_\gamma, \xi^*_\gamma)$. Thus, passing to the weak limit, we find that there exists a non-empty open interval included in $(0, +\infty)$ on which $N_{HS}'\leq -c^{-1} (1- c^{-1})$. This contradicts the explicit expression of $N_{HS}'$, namely $N_{HS}'=-c^{-1}(1-c^{-1})e^{-\xi/c}$; and therefore $\lim_{\gamma\to \infty} \xi^*_\gamma=0.$

The convergence of $\Ng(\xi^*_\gamma)$ towards $1-c^{-1}$ follows from the same arguments as the one of $\Ng^0$: we note that
\[
J_\gamma(\xi^*_\gamma) = c \Ng(\xi^*_\gamma) - \gamma \frac{1}{c}\left(1-\frac{1}{c}\right)  \Ng(\xi^*_\gamma)^\gamma.
\]
Since $\limsup_{\gamma\to \infty} \Ng(\xi^*_\gamma)<1$, the second term in the right-hand side converges towards zero exponentially fast. We also recall that by uniform convergence of $J_\gamma$, $J_\gamma(\xi^*_\gamma)\to J_{HS}(0)=c-1$ as $\gamma\to \infty$. Hence $\lim_{\gamma\to \infty} \Ng(\xi^*_\gamma) = 1- c^{-1}$.

\end{proof}

\subsection{Quantitative bounds for the profiles $\Ng$ }

In order to prove our quantitative stability result in Theorem~\ref{thm:stab-Ng}, we will need some quantitative information on the asymptotic behavior of $\Ng$ and its derivatives (e.g., the size of $\| \Ng'\|_{L^{\infty}}$). 
This subsection is devoted to the proof of such bounds. 
More precisely, we prove the following result:

\begin{lem}
    \label{lem:quantitative-bounds}
    There exists a constant $C>1$, depending only on $c$, such that the following properties hold, for any $\gamma>0$:
       \[\ba
\sup_{0<|h|\leq 1}\sup_{x\in \R} \frac{1}{|h|} \frac{|\Ng (x+ h) - \Ng(x)|}{\Ng(x)} + \left\|\frac{\Ng'}{\Ng}\right\|_\infty \leq C^\gamma,\\
\ \sup_{0<|h|\leq 1}\sup_{x\in \R} \frac{1}{|h|} \left|\frac{\Ng (x+ h) - \Ng(x)}{\Ng'(x)}\right| \leq C^\gamma,\
\sup_{\xi<0}\left| \frac{1-\Pg}{\Pg'}\right| \leq C.
\ea\]

\end{lem}

\begin{proof}

\noindent\textbf{Bound on $\Ng'/\Ng$ in the free zone $\xi>\xi^*_\gamma$.}

We set $L_\gamma:=\frac{\Ng'}{\Ng} + c^{-1}$. Using the equation and the convexity of $\Ng$ in $\xi>\xi^*_\gamma$, see \eqref{conv-free-zone}, we have
\[
-c \Ng'(\xi) \geq \Ng(\xi) (1- (\Ng(\xi))^\gamma)
\implies L_\gamma(\xi) = \frac{\Ng'(\xi)}{\Ng(\xi)} + \frac{1}{c} \leq \dfrac{(\Ng(\xi))^\gamma}{c} \leq \dfrac{(\Ng(\xi^*_\gamma))^\gamma}{c}
\quad \forall \ \xi \geq \xi^*_\gamma.
\]
Furthermore, since $\Ng'(\xi^*_\gamma)= -c^{-1} (1-c^{-1})$ and $\Ng(\xi^*_\gamma)\to 1-c^{-1} $, we immediately infer that $\Lg(\xi^*_\gamma)$ vanishes as $\gamma\to \infty$. We now derive an equation for $\Lg$ in order to obtain a lower bound on $\Lg$.
We have, using the equation on $\Ng$,
\begin{eqnarray*}
\Lg'&=&\frac{\Ng''}{\Ng} - \frac{(\Ng')^2}{\Ng^2}\\
&=&-\frac{1}{\gamma \Ng^{\gamma+1}}\left( \Ng (1-\Ng^\gamma) + c \Ng' + \gamma^2 (\Ng')^2 \Ng^{\gamma-1}\right) - \frac{(\Ng')^2}{\Ng^2}\\
&=& -\frac{c\Lg}{  \gamma \Ng^\gamma } + \frac{1}{\gamma} - (\gamma^2 +1) \left(\Lg-\frac{1}{c}\right)^2. 
\end{eqnarray*}
Thus $\Lg$ satisfies the differential equation
\[
\Lg' + \left[ (\gamma^2 + 1) \Lg  + \frac{c}{\gamma \Ng^\gamma}  -\frac{2(\gamma^2+1)}{c} \right] \Lg = \frac{1}{\gamma} - \frac{\gamma^2 + 1}{c^2}.
\]
Note that the coefficient $\frac{c}{\gamma \Ng^\gamma}  -\frac{2(\gamma^2+1)}{c}$ is exponentially large in the free zone, and  drives a strong convergence of $\Lg$ towards zero. Thus the whole idea is to prove that the quadratic term $(\gamma^2+1)\Lg^2$ does not perturb the linear behavior. This easily follows from a bootstrap argument. Indeed, note that at $\xi=\xi^*_\gamma$, for $\gamma$ large enough
\be\label{in:bootstrap-Lg}
(\gamma^2 + 1) \Lg  + \frac{c}{\gamma \Ng^\gamma}  -\frac{2(\gamma^2+1)}{c}> \frac{c}{2\gamma \Ng^\gamma}.
\ee
Thus by continuity, this property remains true on a non-empty open interval on the right of $\xi^*_\gamma$. Let
\[
\xi_\text{max}:=\sup\{\xi>\xi^*_\gamma,\ \eqref{in:bootstrap-Lg}\text{ holds  on } (\xi^*_\gamma, \xi)\}.
\]
Then $\xi_\text{max}>\xi^*_\gamma$, and on the interval $(\xi^*_\gamma, \xi_\text{max})$, we have
\[
\Lg' + \frac{c}{2\gamma \Ng^\gamma} \Lg \geq \frac{1}{\gamma} - \frac{\gamma^2 + 1}{c^2}\geq -\frac{\gamma^2 + 1}{c^2}.
\]
The Gronwall Lemma then implies that for all $\xi\in (\xi^*_\gamma, \xi_\text{max})$,
\[
\Lg(\xi)\geq \Lg(\xi^*_\gamma)\exp\left(-\int_{\xi^*_\gamma}^\xi \frac{c}{2\gamma \Ng^\gamma}\right) - \frac{\gamma^2 + 1}{c^2}\int_{\xi^*_\gamma}^\xi \exp\left(-\int_{\xi'}^\xi \frac{c}{2\gamma \Ng^\gamma}\right) d\xi'.
\]
Now, we recall that for $\xi>\xi^*_\gamma$, for $\gamma$ sufficiently large,
\[
\Ng(\xi)\leq \Ng(\xi^*_\gamma)\leq 1 - \frac{1}{2c}.
\]
Thus for all $\xi\in (\xi^*_\gamma, \xi_\text{max})$,
\begin{eqnarray*}
\Lg(\xi) &\geq & -|\Lg(\xi^*)| \exp\left(-(\xi-\xi^*_\gamma) \frac{c}{2\gamma}\left(1-\frac{1}{2c}\right)^{-\gamma}\right) \\
&&- \frac{\gamma^2+1}{c}\int_{\xi^*_\gamma}^\xi \exp\left(-(\xi-\xi') \frac{c}{2\gamma}\left(1-\frac{1}{2c}\right)^{-\gamma}\right)d\xi'\\
&\geq &  -|\Lg(\xi^*)|  \exp\left(-(\xi-\xi^*_\gamma) \frac{c}{2\gamma}\left(1-\frac{1}{2c}\right)^{-\gamma}\right) - \frac{2\gamma(\gamma^2 +1)}{c^2}\left(1-\frac{1}{2c}\right)^\gamma.
\end{eqnarray*}
Note that the right-hand side of the above inequality converges uniformly towards zero. In particular, for $\gamma$ sufficiently large, $\Lg(\xi)\geq -1$  for all $\xi\in (\xi^*_\gamma, \xi_\text{max})$. It follows that
\[
(\gamma^2 + 1) \Lg  + \frac{c}{\gamma \Ng^\gamma}  -\frac{2(\gamma^2+1)}{c}\geq \frac{c}{\gamma \Ng^\gamma} - \frac{(c+2)(\gamma^2+1)}{c} \geq \frac{3c}{4\gamma \Ng^\gamma}\quad \forall \xi \in (\xi^*_\gamma, \xi_\text{max}).
\]
By a bootstrap argument, we deduce that $\xi_\text{max}=+\infty$. This implies, in particular, that $\Lg\to 0$ uniformly on $(\xi^*_\gamma, +\infty)$.

\begin{rmk}
The uniform convergence of $\Lg$ towards zero yields the existence of sub-solutions of $\Ng$ in the zone $\xi>\xi^*_\gamma$. 
Indeed, let $\delta>0$ be arbitrary. Then for $\gamma$ large enough, $\Lg\geq - \delta$, and therefore $\frac{\Ng'}{\Ng}\geq -(c^{-1}+\delta)$. By the Gronwall Lemma, we obtain
\be\label{sub-sol}
\Ng(\xi)\geq \Ng(\xi^*_\gamma) \exp\left(-\left(\frac{1}{c} + \delta\right) (\xi-\xi^*_\gamma)\right).
\ee
\label{rmk:sub-sol_free_zone}
\end{rmk}

\noindent\textbf{Bound on $\Ng'/\Ng$ and on the first difference quotient in $L^\infty$.}

We distinguish between $\xi<\xi^*_\gamma$ and $\xi>\xi^*_\gamma$ and we write, for $\gamma$ sufficiently large,
\begin{eqnarray*}
\left\| \frac{\Ng'}{\Ng}\right\|_\infty&=&\max\left(\sup_{\xi<\xi^*_{\gamma}} \frac{|\Ng'|}{\Ng}, \sup_{\xi>\xi^*_\gamma} \left| \Lg - \frac{1}{c}\right| \right)\\
&\leq & \max \left( \frac{1}{\Ng(\xi^*_\gamma)}\|\Ng'\|_\infty, \frac{1}{c}+1\right)\\
&\leq & C | Q_-(\Ng^0)|\leq \left(1-\frac{1}{2c}\right)^{-\gamma}.
\end{eqnarray*}

Let us now consider the difference quotient
\[
\frac{1}{|h|}\frac{|\Ng(x+h)- \Ng(x)|}{\Ng(x)}.
\]
We will need to distinguish several cases:
\begin{itemize}
    \item Case $x<\xi^*_\gamma$: in that case, $\Ng(x)\geq \Ng(\xi^*_\gamma)\to 1-c^{-1}$, and therefore the difference quotient is bounded by $C\|\Ng'\|_\infty$.
    
    \item Case $x>\xi^*_\gamma$:
    \begin{itemize}
        \item Sub-case $h>0$: we write
$\Ng(x+h)-\Ng(x)=\int_0^h \Ng'(x+y)\:dy$, and we recall that since $\Lg$ is uniformly bounded, $|\Ng'|\leq C \Ng$ for some constant $C$ in $(\xi^*_\gamma, +\infty)$. Using the monotony of $\Ng$, we deduce that the difference quotient is bounded.

\item Sub-case $h<0$ and  $x+h>\xi^*_\gamma$: an argument similar to the sub-case $h>0$ applies. In that case, we obtain, using a variant of Remark \ref{rmk:sub-sol_free_zone},
\[
\frac{1}{|h|}\frac{|\Ng(x+h)- \Ng(x)|}{\Ng(x)}\leq C \frac{\Ng(x+h)}{\Ng(x)} \leq C.
\]

\item Sub-case $x+h\leq \xi^*_\gamma$: in that case, note that $x=x+h-h \leq \xi^*_\gamma + 1$ since $|h|\leq 1$. Hence $\Ng(x)\geq \Ng(\xi^*_\gamma+1)$, which is uniformly bounded from below thanks to \eqref{sub-sol}. Thus the difference quotient is bounded by $C \|\Ng'\|_\infty$.
        
    \end{itemize}
    
\end{itemize}
Gathering these results, we obtain the bounds announced in the Lemma.

\medskip

\noindent{\bf Bound on $(1-\Pg)/\Pg'$ on $\R_-$.}

Let $M_\gamma:=(1-\Pg)/\Pg'$. According to Proposition \ref{prop:qualitative-Pg-Ng}, $M_\gamma\to (1-P_{HS})/P_{HS}'= -1$ locally uniformly on $\R_-$. So, for $\gamma$ sufficiently large, $M_\gamma(\xi) \in [-3/2, -1/2]$ for all $\xi \in [-1,0]$.
Furthermore we know that
\[
\dfrac{\Ng'}{\Pg'}(\xi) = \dfrac{1}{\gamma (\Ng(\xi))^{\gamma-1}}  \to \dfrac{1}{\gamma}, \quad 
\lim_{\xi\to - \infty}\dfrac{1-\Pg(\xi)}{1-\Ng(\xi)} = \lim_{N\to 1^-} \frac{1-N^\gamma}{1-N} = \gamma ,
\]
so that, thanks to Theorem \ref{t:gilding},
\[
M_\gamma (\xi)= \frac{\Ng'(\xi)}{\Pg'(\xi)} \frac{1-\Ng(\xi)}{\Ng'(\xi)}\frac{1-\Pg(\xi)}{1-\Ng(\xi)} \to -\left(\sqrt{1 + \frac{c^2}{4\gamma^2}} - \frac{c}{2\gamma}\right)^{-1}, \quad \text{as} ~\xi \to -\infty.
\]
Now, let us consider the interval $(-\infty, -1]$. There are two possibilities:
\begin{itemize}
    \item either $M_\gamma(\xi)\in [M_\gamma(-1), M_\gamma(-\infty)]$ for all $\xi \in (-\infty, -1]$. In that case, for $\gamma$ sufficiently large, $M_\gamma(\xi) \in [-3/2, -1/2]$ for all $\xi \in(-\infty, -1]$;
    
    \item or $M_\gamma$ takes values outside the interval  $[M_\gamma(-1),M_\gamma(-\infty) ]$. In that case $M_\gamma$ reaches a local extremum at some $\xi_M\in (-\infty, -1)$, and therefore $M_\gamma'(\xi_M)=0$.
    
    Let us compute $M_\gamma'$. Using the equation satisfied by $\Pg$~\eqref{eq:Pg}, we have
   \begin{eqnarray*}
   M_\gamma'&=& -1 - \frac{\Pg'' (1-\Pg)}{(\Pg')^2}\\
   &=&-1 + \frac{1-\Pg}{(\Pg')^2}\left( 1-\Pg + \frac{c\Pg'}{\gamma \Pg} + \frac{(\Pg')^2}{\gamma \Pg}\right)\\
   &=&-1 + M_\gamma^2 + c \frac{M_\gamma}{\gamma \Pg} + \frac{1-\Pg}{\gamma \Pg}.
    \end{eqnarray*}
    At $\xi=\xi_M$, the right-hand side vanishes, and therefore
    \[
    M_\gamma(\xi_M)=\frac{1}{2}\left(-\frac{c}{\gamma \Pg(\xi_M)} \pm \sqrt{4 + \frac{c^2}{\gamma^2 \Pg(\xi_M)^2} - 4 \frac{1-\Pg(\xi_M)}{\gamma \Pg(\xi_M)}}\right).
    \]
    Note that, thanks to~\eqref{eq:encadr_P_R-}, $\Pg(\xi_M)\geq \Pg(-1)\geq 1 - e^{-\lambda} >0$. Hence $M_\gamma(\xi_M)= \pm 1 + O(\gamma^{-1})$. Recalling that $M_\gamma<0$ on $\R_-$, we deduce that $M_\gamma(\xi_M)=-1 + O(\gamma^{-1})$.
    
    Once again, for $\gamma$ sufficiently large, we find that $M_\gamma(\xi) \in [-3/2, -1/2]$ for all $\xi \in(-\infty, -1]$.

\end{itemize}

Hence in all cases, we deduce that for $\gamma $ sufficiently large,
\be\label{P'/1-P}
-\frac{3}{2}\leq \frac{1-\Pg}{\Pg'}\leq -\frac{1}{2}\quad \forall \xi \in \R_-. 
\ee

Note that these bounds (which are stronger than what is announced in the statement of the Lemma) imply in particular the following inequalities, which are easy consequences of the Gronwall Lemma: for all $\xi\leq \xi'\leq 0$, for $\gamma$ large enough,
\be\label{pointwise-Pg}
(1-\Pg(\xi))\exp\left(-2(\xi'-\xi)\right)\leq 1-\Pg(\xi') \leq (1-\Pg(\xi)) \exp\left(-\frac{2}{3}(\xi'-\xi)\right).
\ee

\medskip

\noindent{\bf Bound on the second difference quotient.}

We now address the bound on
\[
\sup_{0<|h|\leq 1}\sup_{x\in \R} \frac{1}{|h|} \left|\frac{\Ng (x+ h) - \Ng(x)}{\Ng'(x)}\right|.
\]
Once again, we will need to distinguish between several zones. First, note that
\[
\frac{1}{|h|} \left|\frac{\Ng (x+ h) - \Ng(x)}{\Ng'(x)}\right| = \frac{1}{|h|} \left|\frac{\Ng (x+ h) - \Ng(x)}{\Ng(x)}\right| \; \left| \frac{\Ng(x)}{\Ng'(x)}\right| .
\]
Hence for $x>-2$, this difference quotient is bounded by
\[
\sup_{x\in \R} \sup_{0<|h|\leq 1}\frac{1}{|h|} \left|\frac{\Ng (x+ h) - \Ng(x)}{\Ng(x)}\right| \ \sup_{x>-2}\frac{\Ng(x)}{|\Ng'(x)|}.
\]
For $x>\xi^*_\gamma$, $\Ng/\Ng'= (\Lg- c^{-1})^{-1}$, and we recall that $\Lg$ converges uniformly towards zero on $(\xi^*_\gamma, +\infty)$. Hence $\Ng/\Ng'$ is uniformly bounded on $(\xi^*_\gamma, +\infty)$. 
And looking at the variations of $\Ng'$, we infer that
\[
\sup_{x\in (-2, \xi^*_\gamma)}\frac{\Ng(x)}{|\Ng'(x)|}\leq \max \left(\frac{1}{|\Ng'(-2)|}, \frac{1}{|\Ng'(\xi^*_\gamma)|}\right)\leq C\gamma.
\]
Thus
\[
\sup_{0<|h|\leq 1}\sup_{x\in (-2, +\infty)} \frac{1}{|h|} \left|\frac{\Ng (x+ h) - \Ng(x)}{\Ng'(x)}\right|\leq \gamma C^\gamma \leq C_1^\gamma,
\]
for some constant $C_1>C.$

We now consider the interval $(-\infty, -2)$. Since $|h|\leq 1$, we have $x+h\leq -1$. Hence $x$ and $x+h$ are in the congested zone.
We write
\[
\frac{1}{h}\frac{\Ng(x+h)-\Ng(x)}{\Ng'(x)}=\int_0^1\frac{\Ng'(x+\tau h)}{\Ng'(x)}\:d\tau.
\]
Recall that $\Ng'=\gamma^{-1} \Pg'\Ng^{-(\gamma-1)}.$ Hence
\[
\frac{\Ng'(x+\tau h)}{\Ng'(x)}= \frac{\Pg'(x+\tau h)}{\Pg'(x)}\; \frac{\Ng(x)^{\gamma-1}}{\Ng(x+h)^{\gamma-1}}.
\]
Note that $\Ng^{\gamma-1}=\Pg/\Ng$ is uniformly bounded from above and from below on $(-\infty, -1)$. Thus we focus on the quotient $\Pg'(x+\tau h)/\Pg'(x)$, which we further decompose as
\[
\frac{\Pg'(x+\tau h)}{\Pg'(x)}= \frac{\Pg'(x+\tau h)}{1-\Pg(x+\tau h)}\; \frac{1-\Pg(x+\tau h)}{1-\Pg(x)}\; \frac{1-\Pg(x)}{\Pg'(x)}= \frac{M_\gamma(x)}{M_\gamma (x+\tau h)} \; \frac{1-\Pg(x+\tau h)}{1-\Pg(x)}.
\]
Using \eqref{pointwise-Pg} and \eqref{P'/1-P}, we deduce that
\[
\left| \frac{\Pg'(x+\tau h)}{\Pg'(x)}\right| \leq C e^{2|h|}.
\]
Hence
\[
\sup_{0<|h|\leq 1}\sup_{x\leq -2} \frac{1}{|h|} \left|\frac{\Ng (x+ h) - \Ng(x)}{\Ng'(x)}\right|\leq C.
\]


\end{proof}

Our nonlinear stability result will hold in weighted Sobolev spaces. 
The weights will depend on the function $\Ng$ and its derivative, and therefore will have abrupt changes in the transition zone $(0, \xi^*_\gamma)$. 
In order to monitor precisely these changes, we introduce two additional abscissa $\xim$ and $\txg$, which we define as follows:

\begin{df}[Definition of $\xim$ and $\txg$]~~  
\label{def:xim-txg}
\begin{itemize}
    \item The abscissa $\xim\in \R$ is the unique point where 
    \begin{equation}
     \Pg(\xim)= \left(\dfrac{c^3}{(c-1)(\gamma+1)}\right)^{1/2}.   
    \end{equation}
    
    \item The abscissa $\txg\in \R$ is the unique point such that $\Ng(\txg)\in (0, \Ng^0)$ and
    \[
    \Ng'(\txg)= -\frac{c-1}{4\gamma^2 \Ng(\txg)^{\gamma-1}}.
    \]
    
\end{itemize}

\end{df}

\begin{rmk}
\begin{itemize}
    \item Note that $\xim$ is well-defined by monotony of $\Pg$, and $\xim<0$ since $\Pg(\xim)>\Pg(0)$;
    
    \item The definition of $\txg$ is a little more intricate. We recall that for all $\Ng\in (0, \Ng^0)$, $Q_-(N)<\Ng'< 0$, where $Q_-$ is defined in \eqref{def:Q_pm} and $d\Ng'/d\Ng\leq 0$ for all $\Ng\in (0, \Ng^0)$; we refer to the analysis of the phase portrait in the previous subsection. 
    
    Now, define $\tilde Q(N)$ by
    \[
    \tilde Q(N):=-\frac{c-1}{4 \gamma^2 N^{\gamma-1}}.
    \]
    It is clear from the definition of $\tilde Q$ and $Q_-$ that $Q_-<\tilde Q$ for all $N\in (0, \Ng^0)$, and $\tilde Q$ is monotone increasing on that interval. 
    Consequently, the curve $(\Ng, \Ng')$ intersects the curve $(N, \tilde Q(N))$ exactly once on the interval $(0, \Ng^0)$ (see Figure~\ref{fig:phaseplane-2}). 
    We denote the abscissa of the intersection point as $\tilde \Ng$, and $\txg$ is defined implicitly as $\Ng(\txg)=\tilde \Ng$. 
    \begin{figure}[ht]
\begin{center}
\includegraphics[scale=0.45]{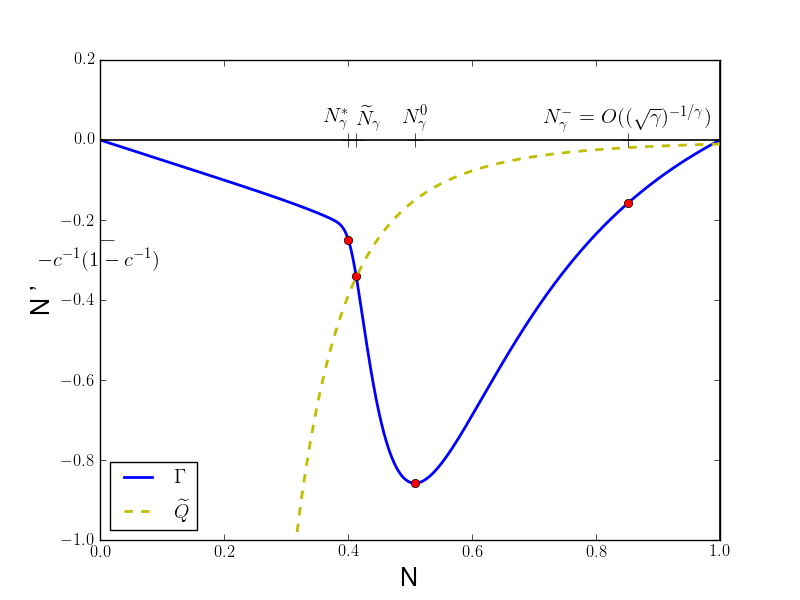}
\end{center}
\caption{{\label{fig:phaseplane-2}Definition of the point $\widetilde{N}_\gamma$ in the phase plane $(N,N')$, $c=2$, $\gamma=5$.} 
}
\end{figure}
    
\end{itemize}

\end{rmk}

Let us now give some properties of $\txg $ and $\xim$, which will be used in the next section:

\begin{lem}[Properties of $\txg $ and $\xim$]~
For $\gamma$ large enough, the following properties hold:
    \begin{itemize}
    \item $\xi^-_\gamma < 0 < \xi^0_\gamma<\txg<\xi^*_\gamma$. As a consequence, $\lim_{\gamma\to \infty} \Ng(\txg)= 1-c^{-1}$;
    
    \item $\xim=O(\gamma^{-1/2})$, and $\txg=O(\gamma^{-1})$;
    
    \item $\Pg'\leq -C \gamma^{-1}$ for all $\xi \in (\xim, \txg)$.
    
    \end{itemize}
    \label{lem:txg-xim}
\end{lem}

\begin{proof}
{\bf Relative positions of $\xi^0_\gamma, \txg, \xi^*_\gamma$.}

By definition of $\txg$, $\Ng(\txg)<\Ng^0$, and thus $\xi^0_\gamma<\txg$.
Furthermore, we recall that $\Ng' $ is monotone increasing on $(\xi^0_\gamma, +\infty)$, and
\[
\Ng'(\txg)=-\frac{c-1}{4 \gamma^2 \Ng(\txg)^{\gamma-1}}\leq -\frac{c-1}{4 \gamma^2 (\Ng^0)^{\gamma-1}} \to -\infty.
\]
Whence $\Ng'(\txg)<\Ng'(\xi^*_\gamma)$, and therefore $\txg<\xi^*_\gamma$.
The limit of $\Ng(\txg)$ follows from the monotony of $\Ng$ and the fact that $\lim_{\gamma\to \infty} \Ng(\xi^0_\gamma)=\lim_{\gamma\to \infty} \Ng(\xi^*_\gamma)=1-c^{-1} $ (see Lemma \ref{lem:transition}).

\medskip

\noindent{\bf Size of $\xim$.} 

First, considering the sub-solution for $\Pg$, we see that $\xim>-1$.
Using \eqref{P'/1-P}, we recall that  $\Pg'$ is bounded away from zero on $(-1,0)$, for $\gamma$ large enough. Thus,
\[
 \frac{|\Pg(\xim)-\Pg(0)|}{\sup_{(\xim, 0)}|\Pg'|}\leq |\xim|\leq \frac{|\Pg(\xim)-\Pg(0)|}{\inf_{(\xim, 0)}|\Pg'|},
\]
and thus
\[
\frac{C^{-1}}{\sqrt{\gamma}}\leq |\xim|\leq \frac{C}{\sqrt{\gamma}}.
\]

\medskip

\noindent{\bf Lower-bound for $|\Pg'|$ on $(\xim, \txg)$  and size of $\txg$.}

Let us introduce yet another intermediate point $\xiint$ such that $\Ng(\xiint)=1- (2c)^{-1}.$ We recall that $\Ng(\xi^0_\gamma)\to 1 -c^{-1}$, and therefore $\xiint\in (0, \xi^0_\gamma)$ for $\gamma$ large enough. Now, for $\xi\in (\xim, \xiint)$, we have $\Ng(\xi)\in [1-(2c)^{-1}, 1]$, and 
\[
\Pg'=\frac{J_\gamma - c\Ng}{\Ng}\leq J_\gamma - c + \frac{1}{2}.
\]
We recall that $J_\gamma(\xi) \to c-1$ uniformly on that interval. Thus $\Pg'\leq -C<0$ on $(\xim, \xiint)$ for $\gamma$ sufficiently large, for some uniform constant $C$.

In particular, since $\Pg(\xiint)= (1-(2c)^{-1})^\gamma$ is exponentially small, it follows that
\[
\xiint\leq \frac{|\Pg(\xiint) - \Pg(0)|}{\inf_{[0,\xiint]} |\Pg'|} \leq \frac{C}{\gamma}.
\]

Let us now consider the intervals $(\xiint, \xi^0_\gamma)$ and $(\xi^0_\gamma, \txg)$.
Using the notations introduced in the previous subsection, it is easily checked that $\Ng(\xiint)\leq N_1$. As a result, using the phase portrait of $\Ng$ (see Figure~\ref{fig:phaseplane-2}), $\Ng'\leq Q_-(\Ng)$ for all $\xi \in (\xiint, \xi^0_\gamma)$. In particular,
\begin{align*}
\Pg'
& = \gamma \Ng'(\xi) \big(\Ng(\xi)\big)^{\gamma-1} \\
& \leq \gamma \big(\Ng(\xi)\big)^{\gamma-1} \times \dfrac{1}{2\gamma^2 \big(\Ng(\xi)\big)^{\gamma-1}} \big(-c -\sqrt{c^2 -4\gamma^2 \big(\Ng(\xi)\big)^{\gamma}(1-\big(\Ng(\xi)\big)^{\gamma})} \big) \\
& \leq -\frac{c}{2\gamma} \qquad \forall \xi \in (\xiint, \xi^0_\gamma).
\end{align*}
For $\xi \in (\xi^0_\gamma, \txg)$, the argument is similar. On this interval, $\Ng'\geq Q_-(\Ng)$, but $\Ng'\leq \tilde Q(\Ng)$ by definition of $\txg$. Thus
\[
\Pg'(\xi) \leq \gamma \big(\Ng(\xi)\big)^{\gamma-1} \times \left(-\dfrac{c-1}{4\gamma^2 \big(\Ng(\xi)\big)^{\gamma-1}} \right)
\leq -\frac{c-1}{4\gamma } \quad \forall \xi \in (\xi^0_\gamma, \txg).
\]
We obtain the desired lower bound on $|\Pg'|$ on $(\xiint, \txg)$. It follows that
\[
\txg - \xiint\leq \frac{|\Pg(\txg)- \Pg(\xiint)|}{\inf_{(\xiint, \txg)} |\Pg'|} \leq C \gamma \left(1-\frac{1}{2c}\right)^\gamma= o (\gamma^{-1}).
\]
Hence $\txg$ and $\xiint$ are exponentially close. The estimate on $\txg$ follows.

\end{proof}

Let us conclude this section by saying a few words about the proof of Theorem~\ref{thm:prop-Ng}.
The sizes and signs of $\xim$ and $\txg$ are given in Lemma \ref{lem:txg-xim}. Inequality \eqref{encadrement-N-congest} follows from the monotony of $\Ng$ and from the definition of $\xim$.
Let us say few words about the inequality claimed in~\eqref{encadr-P-xim}.
Actually, the reader may check that the derivation of sub- and super-solutions on $\R_-$ made in~\eqref{est:superPg}-\eqref{est:subPg} can be easily adapted to the interval $(-\infty,\xim]$, using the fact that $\Pg(\xim) = O(\gamma^{-1/2})$ and $\gamma \Pg \geq C \sqrt{\gamma}$ on $(-\infty,\xim]$.
It follows that
\[
1 - \left(1 - \Pg(\xim)\right) e^{\mu_\gamma \xi}\leq \Pg(\xi) \leq 1 - \left(1 - \Pg(\xim)\right) e^{\xi},
\]
where $\mu_\gamma$ is the positive root of $\mu^2 + \frac{C}{\sqrt{\gamma}} \mu -1=0$. It is easily checked that $\mu= 1- O (\gamma^{-1/2})$, which leads to inequality \eqref{encadr-P-xim}.

The size of $\|\Ng'\|_\infty$ in the intermediate region $(\xim, \txg)$ is an easy consequence of Lemma \ref{lem:transition}, and the bounds on the pressure in that zone follow from the monotony of $\Pg$, the definitions of $\xim$ and the asymptotic behavior of $\Ng(\txg)$ (see Lemma \ref{lem:txg-xim}).

Eventually, the lower and bounds on $\Ng$ in the free zone follow from \eqref{sub-sol}  and \eqref{super-sol-Ng} respectively.

The convergence properties for $\Ng, \Pg$ at the end of Theorem \ref{thm:prop-Ng} are a consequence of Lemma \ref{lem:transition} and Proposition \ref{prop:qualitative-Pg-Ng}.


\section{Stability of the profiles $N_\gamma$}{\label{sec:stability}}

The goal of this section is to prove that the solution of the equation
\be
\label{eq:ng}
\partial_t \ng - \gamma \partial_x \big(\ng^\gamma \partial_x \ng)\big) = \ng \big(1-\ng^\gamma\big)
\ee
associated to an initial datum that lies between two shift of the profile $\Ng$, converges (in a sense specified below) towards $\Ng$ as $t \to +\infty$.
After a presentation of the general strategy, we enter in the details of the two main steps of the demonstration: the analysis of the linearized system and, next, the control of the nonlinear contributions. To keep the presentation as seamless as possible, we have postponed the proof of some technical lemmas to the next section.

This section contains rather technical ingredients. Therefore, in order to alleviate the notation as much as possible, {\it we will systematically drop the dependency with respect to $\gamma$} in the computations and proofs: $\Ng$ will be denoted by $N$, $\ng$ will be denoted by $n$, etc.
We only keep track of this dependency in the statement of our main result.

In the whole section, for all weights and coefficients $f(t,x)$ that only depend on $\xi = x-ct$, we denote $f(t,x)= \bar f(x- ct)$.



\subsection{Overall strategy}

We define here our notion of stability and convergence towards the profile $\Ng$. We introduce a weight 
\begin{equation}\label{formula-weight}
\bar w_0(\xi) =K \Ng^\gamma(\xi) ( \Ng'(\xi))^2\exp\left(\int_{\xim}^\xi \frac{c}{\gamma \Ng(z)} dz \right),
\end{equation}
with a normalization constant $K$ chosen so that $\bar w_0(\xim)=1$. 
We will prove that for sufficiently small and decaying initial data,
	\begin{equation}
	\int_{\R} \Big|\dfrac{\ng(t,x) -\Ng(x-ct)}{\Ng'(x-ct)}\Big|^2 \bar w_0 (x-ct) \ dx \to 0 \qquad \text{as} \quad t\to +\infty.
	\end{equation}
The result is summarized in the following theorem.	
\begin{thm}\label{prop:stab}
There exists $\eta\in ]0,1[$ such that the following result holds. Let $\gamma > 1$ be fixed, sufficiently large.
Let us assume that $n^0_\gamma$ lies between two shifts of $\Ng$, {\it i.e.} there exists $h >0$ such that $n^0_\gamma(x) \in [\Ng(x+h),\Ng(x-h)]$ for all $x\in \R$.
Let $\ng$ be the solution of~\eqref{eq:pde-0} associated with $\ng^0$ and 
\[
u_\gamma(t,x) := \dfrac{\ng(t,x) - \Ng(x-ct)}{\Ng'(x-ct)}.
\]
Assume that 
\[
\int_\R |u_\gamma(0,x)|^2 \bar w_0(x) \:dx<+\infty.
\]

Then there exists a constant $c_\gamma > 0$, decreasing exponentially with $\gamma$, such that if $h\leq \eta^\gamma$, the following inequalities hold
\begin{equation}
\ba
\int_\R |u_\gamma(t,x)|^2 \bar w_0(x-ct) \ dx 
\leq e^{-c_\gamma t} \int_\R |u_\gamma(0,x)|^2 \bar{w_0}(x) \ dx \qquad \forall t \geq 0,\\
\gamma \int_0^\infty \int_\R |\p_x u_\gamma(t,x)|^2 \Ng^\gamma(x-ct) \bar w_0(x-ct)\:dx\:dt\leq  \int_\R |u(0,x)|^2 \bar{w_0}(x) \ dx.
\ea
\end{equation}
\end{thm}	
	Note that this statement is merely a rephrasing of Theorem \ref{thm:stab-Ng} in terms of the unknown $u_\gamma$. We emphasize that $u_\gamma$ is a natural variable when linearizing equation \eqref{eq:ng} around $\Ng(x-ct)$.
	Indeed, since equation \eqref{eq:ng} has constant coefficients and since $\Ng(x-ct)$ is a particular solution of the equation, it is classical that $\p_x \Ng(x-ct)$ is a solution of the linearized equation around $\Ng(x-ct)$ (and we also recall that $\p_x \Ng$ does not vanish on $\R$).
	Moreover, $\ng(t,x)-\Ng(x-ct)$ is also a solution of the linearized equation, up to a quadratic remainder which we will treat perturbatively.
	Therefore working with energies depending on $u_\gamma$ is similar to deriving relative entropies for the system.
	

	The result relies on two main estimates: a $L^\infty$ control on $n-N$ (almost immediate, see below) and a more complicated $L^2$ weighted estimate on the variable $u$.
	Indeed, an easy computation (see subsection \ref{ssec:eq-u}) shows that $u$
	satisfies the equation
	\begin{equation}{\label{eq:u}}
	\partial_t u + b \partial_x u - a \partial^2_x u 
	= \dfrac{\gamma}{\gamma+1} \dfrac{\partial^2_x G(u)}{N'(x-ct)} - \dfrac{G(u)}{N'(x-ct)},
	\end{equation}
	with $a=\bar a(x-ct)$, $b=\bar b(x-ct)$ and 
	\begin{align*}
	\bar a := \gamma N^\gamma, \quad 
	\bar b := -2\gamma \dfrac{\big(N^\gamma N'\big)'}{N'} 
		= -2\gamma^2 N^{\gamma-1}N' - 2\gamma N^\gamma \dfrac{N''}{N'},
	\end{align*}
	and
	\be\label{def:G}
	G(u) := n^{\gamma+1} - N^{\gamma+1}(x-ct) - (\gamma+1)N^\gamma (x-ct)(n-N(x-ct)).
	\ee
	
	Let us make a few remarks before exposing the main ingredients of the proof. First, we emphasize that all unknowns and coefficients depend on $\gamma$ (i.e. $b,a, u, G, N$). 
	As mentioned above, we chose not to make this dependency explicit in our notation.
	Second, equation \eqref{eq:u} has a structure of the type
	\[
	\p_t u + \mathcal L u= \mathcal G[u],
	\]
	where $\mathcal L$ is a linear operator, corresponding to the linearization of equation \eqref{eq:ng} around $\ng=\Ng$, and $\mathcal G[u]$ is a quadratic operator in the sense of \eqref{in:G}.
	
	Quite classically, the core of our proof relies on the two following observations:
	\begin{itemize}
	    \item The linear operator $\mathcal L$ is coercive in some weighted $H^1$ space. More precisely, there exists a weight $\bar w$ and a constant $\delta=\delta_\gamma>0$ with the following property:
	    for any $v\in \mathcal C^2_c(\R)$,
	    \be\label{in:coercivity}
	    \int_{\R} (\bar b\p_\xi v - \bar a\p_\xi^2 v) v \bar w\geq\int_{\R} (\p_\xi v)^2 \bar a \bar w + \frac{\delta}{2}\int_{\R} |v|^2 e^{\sqrt{\gamma}\xi} - \frac{c}{2} \int_\R |v|^2 \p_\xi \bar w.  
	    \ee
	    Note that the last term will enter the time derivative of the energy $\int |u|^2 w$ when we perform energy estimates.
	    
	    This type of coercivity property had been identified by Leyva and Plaza in \cite{Leyva2020}, without the $L^2$ term $\int_{\R} |v|^2 e^{\sqrt{\gamma}\xi}$, which will play a crucial role in the energy estimates.

	   \item The nonlinear term $\mathcal G[u]$ satisfies
	   \be\label{in:G}
	   |\mathcal G[u]| \leq C_\gamma |u|(|u| + |\p_x u|).
	   \ee
	    Hence, if $\|u\|_{L^{\infty}}$ is small enough, we can hope to absorb this term in the energy dissipation provided by the coercivity of $\mathcal L$.
	    
	\end{itemize}
	The remainder of the section is devoted to a more rigorous statement and to the proofs of the above heuristic arguments.
	Concerning the smallness of the $L^\infty$ bound, a possible strategy could be to differentiate equation \eqref{eq:u} with respect to $x$ and to derive uniform, high regularity bounds on $u$.
	This strategy is likely to succeed. However, it will probably come at a high technical cost.
	Consequently, to simplify the proof and the presentation, we chose here to take advantage of the parabolic structure of the equation and use the comparison principle (or maximum principle), which immediately implies an $L^\infty$ bound on $n$ and $u$.
	
	\begin{rmk}
	Let us mention by anticipation that the constant $\delta$ in \eqref{in:coercivity} will be small, while the constant $C_\gamma$ in \eqref{in:G} will be very large.
	Whence we will need $\|u\|_{L^{\infty}}$ to be very small (in fact, exponentially small) to treat the quadratic term as a perturbation.
	This is related to the strong singularities in $\Ng'$ which were highlighted in the previous section (recall that $\|\Ng'\|_{L^{\infty}}$ blows up exponentially, see Lemmas \ref{lem:transition} and \ref{lem:quantitative-bounds}).

	\end{rmk}

	Let us now present the main ideas of the proof.

\subsubsection*{Structure of the linearized system - weighted $L^2$ estimate}

We  start from a reference weight $\bar w_0$, which is defined as
the solution of the differential equation
	\begin{equation}\label{eq:w}
	\begin{cases}
	(\bar a \bar w_0)'(\xi) + \big(\bar b(\xi) - c\big) \bar w_0= 0 \quad \text{for}~ \xi \in \R,\\
	\bar{w}_0 (\xi_0) = 1 \quad \text{for some}~ \xi_0 \in \R.
	\end{cases}
	\end{equation}
	
This weight is identical to the one identified by Leyva and Plaza in \cite[Section 3.1]{Leyva2020}, although our derivation differs from theirs, see subsection~\ref{sec:proof-lem-3.4}. For this weight $\bar w_0$, we have the following

	\begin{lem}[Stability estimates for the linearized system]\label{prop:lin}
	Let $u$ be a smooth solution to
	\begin{equation}{\label{eq:u-lin}}
	\partial_t u + b \partial_x u - a \partial^2_x u 
	= S, 
	\end{equation}
    where $S$ is a general source term.
The following equality holds, with $w_0(t,x)=\bar w_0(x-ct)$
	\begin{multline}
	  \label{eq:u-lin-0}
	 \int_{\R} |u(t,x)|^2 w_0(t,x) \ dx + 2 \int_0^t\int_{\R} a(s,x)(\partial_x u(s,x))^2 w_0(s,x) \ dx \; ds \\
	= \int_{\R} |u^0(x)|^2 \bar{w}_0(x) \ dx + 2 \int_0^t\int_\R S(s,x) u(s,x) w_0(s,x) \ dx ds.  
	\end{multline}

	\end{lem}

Furthermore the weight $ \bar w_0$ fulfills the following properties:

\begin{lem}[Asymptotic behaviors of $\bar w_0$]\label{lem:w0}
The solution of~\eqref{eq:w} with $\xi_0=\xim$ is given by
\begin{equation}\label{formula-w_0}
\bar w_0(\xi) =K N^\gamma(\xi) ( N'(\xi))^2\exp\left(\int_{\xim}^\xi \frac{c}{\bar a(z)} dz \right), 
\end{equation}
where the normalization constant $K$ is chosen so that $\bar w_0(\xim)=1$. We find that $K\propto \gamma^{3/2}$.
Consequently $\bar w_0$ has the following asymptotic behaviors:
\begin{itemize}
	\item as $\xi \to +\infty$, $\bar w_0$ has a double exponential growth: there exist $C_1, C_2, C>0$ independent of $\gamma$ such that for all $\xi \geq C$,
    \begin{equation}
     \exp\left( \exp\left(C_1\gamma\xi\right)\right)
    \leq \bar w_0 \leq 
     \exp\left( \exp\left(C_2\gamma\xi\right)\right);
    \end{equation}
	
	\item as $\xi \to -\infty$, $\bar w_0$ decreases exponentially to $0$: there exists $C>0$ independent of $\gamma$ such that for all $\xi \leq -C$,
	\begin{equation}
	C^{-1}\frac{K}{\gamma^2}\exp\left(2\left(1+\frac{C}{\sqrt{\gamma}}\right) \xi\right) \leq \bar w_0 \leq C\frac{K}{\gamma^2}\exp\left(2\left(1-\frac{C}{\sqrt{\gamma}}\right) \xi\right) .
	\end{equation}
\end{itemize}

\end{lem} 

\subsubsection*{Spectral gap and Poincaré inequality}

However, the sole weight $w_0$ is not entirely sufficient to have an exponential decay in time of the energy $\int_\R |u|^2 w_0$.
Indeed, in order to prove such  an exponential decay, we need a Poincaré inequality of the type
\[
\int_{\R} |v|^2 \bar w_0 \leq C_\gamma \int_{\R} (\p_x v)^2 \bar a \bar w_0\quad \forall v\in \mathcal C^1_c(\R).
\]
In other words, we need to prove a spectral gap inequality.
To the best of our knowledge, such an inequality does not hold for the weight $\bar w_0$.
However, we are able to prove a variant of such an inequality, with an additional $L^2$ term in the right-hand side:
\begin{prop}[Weighted Poincar\'e-type inequality]\label{prop:poincare}

Let $v\in \mathcal C^1_c(\R)$. 

Then there exists a constant $\bar C$ independent of $\gamma$ and a constant $C_\gamma\leq C^\gamma$ such that
\be \label{eq:poincare-b}
\int_{-\infty}^{\xim} v^2 \gamma \Ng^\gamma \bar w_0 d\xi + \int_{\tilde \xi}^{+\infty} v^2 \frac{1}{\gamma \Ng^\gamma} \bar w_0 d\xi \\
\leq \bar C \int_{\R}   (\partial_\xi v)^2 \bar a \bar w_0 \, d\xi  + C_\gamma \int_{\R} v^2 e^{\sqrt{\gamma}\xi } \ d\xi.
		\ee

In particular, there exists a constant $c_\gamma$, satisfying $c_\gamma \geq \eta^\gamma$ for some $\eta\in ]0,1[$ independent of $\gamma$, such that
\[
c_\gamma \int_\R v^2 \bar w_0 d\xi 
\leq  \int_\R (\partial_\xi v)^2 \bar a \bar w_0 \, d\xi +  \int_\R |v|^2 \exp(\sqrt{\gamma} \xi) d\xi.
\]

	\end{prop}
\begin{rmk}
\begin{itemize}
    \item We recall that we 
    defined $\xim$ so that $\Pg(\xim)=\left(\frac{c^3}{(c-1)(\gamma+1)}\right)^{1/2}$. Consequently, in the first integral of~\eqref{eq:poincare-b}, the term $\gamma \Ng^\gamma$ is bounded from below by $C \sqrt{\gamma}$.
    
    \item In a similar way, for $\xi>\tilde \xi$, we have $\Ng\leq 1- (2c)^{-1}$, so that the term $\frac{1}{\gamma \Ng^\gamma}$ in the second integral in the left-hand side of ~\eqref{eq:poincare-b} is exponentially large.
    
    \item We stated this result for $v\in \mathcal C^1_c(\R)$, but the result can be extended to $v$ in suitable weighted Sobolev spaces by a classical density argument.
    
    \item Let us give a few motivations for the weight $e^{\sqrt{\gamma} \xi}$ in the right-hand side. 
     We actually have some freedom in the choice of  the coefficient of the exponential that we take equal to $\alpha_\gamma=\sqrt{\gamma}$. We could a priori take a larger coefficient $\alpha_\gamma$ with respect to $\gamma$. However $\alpha_\gamma$ must satisfy a number of conditions.
    First, an important feature is that the growth (resp. decay) of this weight as $\xi\to +\infty$ (resp. $\xi\to -\infty$) is lower (resp. stronger) than the one of $\bar w_0$. 
    Moreover, the energy dissipation provides a very good control of the energy in the zone $\xi<\xim$ and $\xi>\txg$, as we can see in inequality \eqref{eq:poincare-b}.
    The additional term $\int_\R v^2 e^{\sqrt{\gamma} \xi}$ is only needed in the transition zone $(\xim , \txg)$, as we shall see in the course of the proof.
   Our choice $\alpha_\gamma=\sqrt \gamma$
    is actually motivated by the need to control, uniformly with respect to $\gamma$, the exponential $\exp(\alpha_\gamma \xi^-_\gamma)$ (see in particular \eqref{phi-3}). Since $\xim = O(\gamma^{-1/2})$, it leads us to set $\alpha_\gamma= \sqrt{\gamma}$.
    
    \item The proof of Proposition \ref{prop:poincare} relies on the quantitative estimations of Lemma \ref{lem:quantitative-bounds}, and will be performed in subsection ~\ref{proof-poincare}.

\end{itemize}

\end{rmk}

As a consequence, if we are able to  have an additional lower-order dissipation term in the energy estimate (the term $\int |v|^2 e^{\sqrt{\gamma } \xi}$ in the right-hand side), the exponential decay of the energy for the linearized system will follow.
In order to get this extra dissipation, it turns out that it is sufficient to modulate slightly the weight $\bar w_0$.
More precisely, we define $\bar w=\bar w_0 \bar \phi$, where
\be\label{eq:psi}\ba
\bar \phi(-\infty)=2,\\
\bar \phi'=-\frac{\delta_\gamma}{\sqrt{\gamma}} \exp\left(\sqrt{\gamma}\xi\right)(\bar a \bar w_0)^{-1},
\ea
\ee
and the constant $\delta_\gamma>0$ is chosen such that $\bar\phi (+\infty)\geq 1$.

Lemma \ref{lem:w0} ensures that $\bar \phi'\in L^1(\R)$, and therefore $\bar \phi$ is well-defined and monotonous. 
Note that since $1\leq \bar \phi \leq 2$ by construction, the weights $\bar w_0$ and $\bar w$ are equivalent. However, choosing $\bar w$ gives us the following additional control:

\begin{lem}
    \label{lem:mod-weight} Under the same assumptions and notation as in Lemma \ref{prop:lin}, we have
    \begin{multline}\label{eq:u-lin-mod}
	\int_{\R} |u(t,x)|^2 w(t,x) \ dx + 2 \int_{0}^t \int_{\R} (\partial_x u)^2 a w
	+ \delta_\gamma  \int_0^t \int_{\R} |u(s,x)|^2 \exp\left(\sqrt{\gamma}(x-cs)\right)dxds\\
	= \int_{\R} |u^0|^2 \bar{w}\ dx + \int_0^t \int_\R S(s,x)u(s,x) w(s,x) \ dx ds,
	\end{multline}	
    with 
    \be\label{est:delta_gamma}
     \delta_\gamma =\delta_0 \gamma^{-1/2}\ee,
    for some constant $\delta_0$ independent of $\gamma$.

Furthermore, there exists $\eta_1\in ]0,1[$ and a constant $c_\gamma\geq \eta_1^\gamma$ such that for all $v\in \mathcal C^1_c(\R)$,
\[
c_\gamma \int_\R |v|^2 \bar w \leq  2\int_{\R} (\partial_x v)^2 \bar a \bar w
	+ \delta_\gamma  \int_{\R} |v(\xi)|^2 \exp\left(\sqrt{\gamma}\xi\right)d\xi.
\]

\end{lem}

\begin{df}
In the rest of the paper, we set
\[
\cD(t):=2 \int_{\R}(\p_x u(t,x))^2 a(t,x) w(t,x) dx + \delta_\gamma \int_{\R} |u(t,x)|^2 \exp(\sqrt{\gamma}(x-ct))dx,
\]
which is the total dissipation term.

\end{df}

Gathering the above results, we see that any solution of the linearized equation \eqref{eq:u-lin} with $S=0$, with an initial data such that $\int_\R |u_0|^2 \bar w_0<\infty$, decays exponentially (at a rate $c_\gamma$) as $t\to \infty$.

\subsubsection*{$L^\infty$ estimate}

In order to prove that the dynamics of the nonlinear equation \eqref{eq:u} is driven by the linearized part of the equation, and that the nonlinear term in the right-hand side of \eqref{eq:u} can be treated perturbatively, we will need a last ingredient, which is a direct consequence of the comparison principle:

\begin{lem}[$L^\infty$ estimate]\label{lem:u-max}
	Let 
	$h$ be small enough and
	assume that $n^0$ lies between two shifts of the reference profile $N$:
		\[
		N (x+h) \leq n^0(x) \leq N(x-h) .
		\]
		Then, for all $t\geq 0$, for all $x\in \R$,
		\[
	N(x+h-ct)\leq 	n(t,x)\leq N(x-h-ct).
		\]
		
		From Lemma \ref{lem:quantitative-bounds}, we have
	\begin{equation}\label{eq:u-max}
	\|u\|_{L^\infty(\R_+ \times \R)} + \left\|\dfrac{n-N(x-ct)}{N(x-ct)}\right\|_{L^\infty(\R_+ \times \R)} \leq C^{\gamma}h,
	\end{equation}
	where $C$ is a positive constant independent of $\gamma$.
	
\end{lem}

\noindent Equipped with this estimate and the control in $L^\infty$ from Lemma~\ref{lem:u-max}, we can control the nonlinear contributions and deduce an exponential decay of the $L^2$ weighted norm as $t\to +\infty$, as stated in Proposition \ref{prop:stab}. The next subsection is devoted to the control of the nonlinear terms. 
We then give a proof of Theorem \ref{thm:stab-Ng} at the end of section \ref{sec:stability}.



\subsection{Control of the nonlinear terms and long-time behavior}
We now address the proof of Theorem ~\ref{prop:stab} using the tools described above.
Let $u$ be a smooth solution to~\eqref{eq:u}, we get by applying~\eqref{eq:u-lin-mod}:
\begin{align}\label{eq:u-NL}
& \int_{\R} |u|^2 w_0 \phi \ dx + 2 \int_0^t\int_{\R} a(\partial_x u)^2 w_0 \phi \ dx ds
+  \int_0^t\int_{\R} |u|^2 {\delta_{\gamma}\exp(\sqrt{\gamma}(x-cs))} \ dx ds \nonumber \\
& = \int_\R |u_0| \bar w + \dfrac{\gamma}{\gamma+1} \int_0^t\int_{\R} \dfrac{\partial^2_x G(u)}{\partial_x N} u w_0 \phi \ dx ds
-  \int_0^t \int_{\R}\dfrac{G(u)}{\partial_x N} u w_0 \phi \ dx ds.
\end{align}
Observe that the first term of the right-hand side comes from the nonlinear diffusion while the second comes from the reaction term.
We also recall that
\begin{equation*}
G(u) = (N + u \partial_x N)^{\gamma+1} - N^{\gamma+1} - (\gamma+1)N^\gamma u\partial_x N.
\end{equation*}

First, let us estimate $G(u)$.

\begin{lem}\label{lem:estim-G}
Assume that
\[
\left\| u\frac{\p_x N}{N}\right\|_\infty=\left\|\frac{n(t,x)- N(x-ct)}{N(x-ct)}\right\|_\infty \leq \frac{1}{\gamma}.
\]

Then
\begin{align}\label{eq:G-u}
|G(u)| &\leq C\gamma^2 (u\partial_x N)^2 N^{\gamma-1}, \\
|\partial_x G(u)|
& \leq C \gamma^3 \partial_x N N^{\gamma-2} (u\partial_x N)^2 
	+ C \gamma^2 N^{\gamma-1} (u\partial_x N) \ \partial_x (u\partial_x N)\label{eq:dG-u}\\
	&\leq C \gamma^3 \partial_x N N^{\gamma-2} (u\partial_x N)^2 + C \gamma^2 N^{\gamma-1} (\p_x N)^{2} |u| \; |\p_x u| \nonumber  \\
	&\qquad+ C \gamma u^2 |\p_x N|^2 N^{-1} + C\gamma u^2 |\p_x N|, \nonumber
\end{align}
for some constant $C > 0$ independent of $\gamma$.
\end{lem}

\begin{proof}

The first estimate can be easily proved by writing
\[
G(u)=N^{\gamma+1} g\left(\frac{u\p_x N}{N}\right),
\]
where $g(X)= (1+X)^{\gamma+1} -1 - (\gamma+1) X.$ A Taylor expansion at order two close to $X=0$  shows that if $|X|\leq \gamma^{-1}$,
\[
|g(X)| \leq \frac{1}{2}\gamma(\gamma+1) \left(1 + \frac{1}{\gamma}\right)^{\gamma-1}|X|^{2}\lesssim \gamma^2|X|^{2}.
\]
Estimate \eqref{eq:G-u} follows.
We also know by convexity that $G(u) \geq 0$.

For the second estimate, we differentiate $G$ and get
\begin{align*}
\partial_x G(u)
& = (\gamma+1) \partial_x N \Big((N+ u\partial_x N)^\gamma -N^\gamma - \gamma N^{\gamma-1} u\partial_x N \Big) \\
& \quad + (\gamma+1)\partial_x(u\partial_x N) \Big(\big(N+ u\partial_x N\big)^\gamma - N^\gamma\Big).
\end{align*}
Reasoning as before, we infer that
\begin{align*}
|\partial_x G(u)|
& \leq C (\gamma+1) \gamma^2 |\partial_x N| |u\partial_x N|^2 N^{\gamma-2}  + C(\gamma+1)\gamma |\partial_x(u\partial_x N)||u\partial_x N|  N^{\gamma-1}.
\end{align*}
To obtain the last set of inequalities, we use the equation on $N$, and we recall that
\[
\gamma \p_x^2 N N^\gamma = -c \p_x N - \gamma^2 (\p_x N)^2 N^{\gamma-1} - N(1-N^\gamma),
\]
which concludes the proof of the lemma.
\end{proof}

\medskip
\begin{lem}[Control of the nonlinear reaction term]\label{lem:NL1}
There exists a constant $\eta_2\in ]0,1[$ such that if
\[
\left\|\frac{u\p_x N}{N}\right\|_\infty\leq \eta_2^\gamma,
\]
then
the following inequality holds
\begin{equation}
\left| \int_{\R}\dfrac{G(u)}{\partial_x N} u w_0 \phi \ dx \right|
\leq \dfrac{1}{4} \cD.
\end{equation}
\end{lem}

\medskip
\begin{proof}
Using Lemma~\ref{lem:estim-G}, we have \[
\left| \int_{\R}\dfrac{G(u)}{\partial_x N} u w_0 \phi \ dx \right|
 \leq C \gamma^2 \int_{\R} |u|^3 |\partial_x N| N^{\gamma-1}w_0 \phi \ dx,
\]
that we want to absorb in the left-hand side of the equality~\eqref{eq:u-NL} thanks to the diffusion and damping terms:
\[
\cD = 2 \int_{\R} a(\partial_x u)^2 w_0 \phi \ dx 
+ \int_{\R} \delta_{\gamma}\exp(\sqrt{\gamma}\xi) u^2 \ dx .
\]
Recalling Proposition~\ref{prop:poincare}, we observe that it suffices to have
\[
\gamma^2 \left\|\frac{u\p_x N}{N}\right\|_\infty \leq \dfrac{c_\gamma}{4},
\]
which concludes the proof, choosing $\eta_2<\eta_1$.

\end{proof}

\medskip
\begin{lem}[Control of the nonlinear diffusion term]\label{lem:NL2}
	There exists a constant $\eta_2\in ]0,1[$ such that if
	\[
	\|u\|_\infty + \left\|\frac{u\p_x N}{N}\right\|_\infty \leq \eta_2^\gamma, 
	\]
	then the following inequality holds
	\begin{equation}
\left|\int_\R \frac{\p_x^2 (G(u))}{\p_x N} u  w_0 \phi \:dx\right|\leq\frac{1}{4} \cD.
	\end{equation}	
\end{lem}

\medskip
\begin{proof}
Integrating by parts the nonlinear term stemming from the diffusion, we have
\be\label{NL}
\int_\R \frac{\p_x^2 (G(u))}{\p_x N} u  w_0 \phi \:dx= -\int_\R \p_x G(u) \p_x u  \left(\frac{w_0\phi}{\p_x N}\right) - \int_\R \p_x G(u) u  \p_x\left(\frac{w_0\phi}{\p_x N}\right)  . 
\ee

$\bullet$ We first address the first term in the right-hand side of \eqref{NL}, using
 the estimate on $\p_x G(u)$ from Lemma \ref{lem:estim-G}. 
It follows that
\begin{eqnarray*}
\left| \int_\R \p_x G(u) \p_x u  \left(\frac{w_0\phi}{\p_x N}\right) \right| 
&\leq & C\int_\R \gamma^3 N^{\gamma-2} |u|^2 |\p_x u| (\p_x N)^2 w_0 \phi\\
&&+ C \int_\R \gamma^2 |u| |\p_x u|^2 N^{\gamma-1} |\p_x N| w_0 \phi\\
&&+ C \int_\R \gamma |u|^2 |\p_x u|  \frac{|\p_x N|}{N} w_0 \phi\\
&&+ C \int_\R \gamma |u|^2 |\p_x u|  w_0 \phi\\
&=&\sum_{i=1}^4 I_i.
\end{eqnarray*}
We then address each term $I_i$ separately. 
We start with the term $I_2$, for which we simply write
\[
I_2\leq C \gamma\left\| u \frac{\p_x N}{N}\right\|_{L^{\infty}} \int_\R (\p_x u)^2 a w_0 \phi,
\]
which is smaller than $\cD/16$, provided $\|u\p_x N/N\|_\infty \leq (16 C \gamma)^{-1}$.

For all other terms, we first perform a Cauchy-Schwarz inequality. We have, recalling that $\bar a = \gamma N^\gamma$,
\[
\ba
I_1\leq C\left( \int_\R (\p_x u)^2 a w\right)^{1/2}\left(\gamma^5 \int_\R |u|^4 N^{\gamma-4} w_0\phi\right)^{1/2},\\
I_3\leq C\left( \int_\R (\p_x u)^2 a w\right)^{1/2} \left(\int_\R \gamma |u|^4 \frac{(\p_x N)^2}{N^{\gamma+2}} w_0\phi\right)^{1/2},\\
I_4\leq C\left( \int_\R (\p_x u)^2 a w\right)^{1/2}\left( \int_\R \gamma |u|^4 N^{-\gamma} w_0\phi\right)^{1/2}.
\ea
\]
We then bound each integral with $|u|^4$ in the right-hand side by using the Poincaré inequality from Proposition \ref{prop:poincare} and the $L^\infty$ estimate on $u$.
The simplest term is $I_1$, for which we have
\[
I_1\leq C\gamma^{5/2}\|u\|_{L^{\infty}}\left( \int_\R (\p_x u)^2 a w\right)^{1/2}\left( \int_\R |u|^2  w_0\phi\right)^{1/2}\leq C\gamma^{5/2}\|u\|_{L^{\infty}} c_\gamma^{-1/2} \cD.
\]
Concerning the term $I_3$, we have
\[
\int_\R \gamma |u|^4 \frac{(\p_x N)^2}{N^{\gamma+2}} w_0\phi\leq \gamma \left\|u\frac{\p_x N}{N}\right\|_{L^{\infty}}^2 \int_\R |u|^2 \frac{w_0\phi}{N^\gamma}.
\]
Using Proposition \ref{prop:poincare}, we have
\begin{eqnarray*}
\int_\R |u|^2 \frac{w_0\phi}{N^\gamma}&= & \int_{\xi>\tilde \xi} |u|^2 \frac{w_0\phi}{N^\gamma}+ \int_{\xi<\tilde \xi}|u|^2 \frac{w_0\phi}{N^\gamma}\\
&\leq & \gamma \int_\R (\p_x u)^2 a w_0\phi + \gamma C_\gamma \int_{\R} u^2 e^{\sqrt \gamma (x-ct)} dx\\
&&+ \frac{1}{N(\tilde \xi)^\gamma}\int_{\xi<\tilde \xi} |u|^2 w_0 \phi\\
&\leq & C_\gamma' \cD, 
\end{eqnarray*}
for some exponentially large constant $C_\gamma'$. The above inequality also allows us to bound $I_4$.
Thus, provided
\[
\|u\|_{L^{\infty}} \leq \delta \inf\left( c_\gamma^{1/2} \delta^{-5/2} ,(C_\gamma')^{-1/2}\gamma^{-1/2}\right),\quad  \left\|u\frac{\p_x N}{N}\right\|_{L^{\infty}}\leq\delta  (C_\gamma')^{-1/2}\gamma^{-1/2},
\]
for some small constant $\delta$ independent of $\gamma$, we infer that
\[
\left| \int_\R \p_x G(u) \p_x u  \left(\frac{w_0\phi}{\p_x N}\right) \right|\leq \frac{1}{8}\cD.
\]

$\bullet$ Let us now consider the second term in the right-hand side of \eqref{NL}. Computing the weight in the right-hand side and using the definitions of $\phi$, $w$ and the equation satisfied by $N$, we find
\begin{eqnarray*}
\p_\xi\left(\frac{\bar w_0\bar \phi}{\p_\xi N}\right)
&=&\p_\xi \bar \phi \frac{\bar w_0}{\p_\xi N}\\
&&+K \bar \phi \p_\xi (\p_\xi N N^\gamma) \exp\left(\int_{\xim}^\xi \frac{c}{\bar a}\right)\\
&&+ K \bar \phi \p_\xi N N^\gamma\frac{c}{\bar a}\exp\left(\int_{\xim}^\xi \frac{c}{\bar a}\right)\\
&=& -\frac{\delta_\gamma}{\sqrt{\gamma}} \frac{\exp(\sqrt{\gamma} \xi)}{\gamma N^\gamma \p_x N} - \frac{K}{\gamma}\bar \phi N (1-N^\gamma) \exp\left(\int_{\xim}^\xi \frac{c}{\bar a}\right)\\
&=:&W_1 + W_2.
\end{eqnarray*}

We then use the estimate on $\p_x G(u)$ from Lemma \ref{lem:estim-G}, treating $W_1$ and $W_2$ separately.
We have, concerning the terms with $W_1$,
\begin{eqnarray*}
&&\left| \int_\R \p_x G(u) u W_1\right|\\
&\leq & C \frac{\delta_\gamma}{\sqrt{\gamma}} \left[ \int_{\R} \gamma^2 |u|^3 \frac{(\p_x N)^2}{N^2} e^{\sqrt{\gamma} \xi} + \int_\R \gamma \frac{|\p_x N|}{N} |u|^2 |\p_x u|  e^{\sqrt{\gamma} \xi}\right]\\
&+  & C \frac{\delta_\gamma}{\sqrt{\gamma}} \int_{\R} |u|^3 N^{-\gamma} \left( 1 + \frac{|\p_x N|}{N}\right)e^{\sqrt{\gamma} \xi}.
\end{eqnarray*}
Using a Cauchy-Schwarz inequality for the second integral, we get
\begin{eqnarray*}
&&\left| \int_\R \p_x G(u) u W_1\right|\\
&\leq & C {\delta_\gamma}\gamma^{3/2} \left\| u \frac{\p_x N}{N }\right\|_{L^{\infty}} \left\| \frac{\p_x N}{N }\right\|_{L^{\infty}} \int_\R u^2 e^{\sqrt{\gamma} \xi} \\
&&+ C \delta_\gamma  \left( \int_\R (\p_x u)^2 a w\right)^{1/2}\left\| u \frac{\p_x N}{N }\right\|_{L^{\infty}}\left(\int_\R \frac{u^2}{N^\gamma w_0}e^{2\sqrt{\gamma} \xi}\right)^{1/2}\\
&&+ C \frac{\delta_\gamma}{\sqrt{\gamma}}\left( \|u\|_{L^{\infty}} + \left\| u \frac{\p_x N}{N }\right\|_{L^{\infty}}\right)\int_\R \frac{u^2}{N^\gamma} e^{\sqrt{\gamma} \xi}.
\end{eqnarray*}
We then observe that thanks to the growth and decay properties of the weight $w_0$ at $\pm \infty$, there exists a constant $C_\gamma$ such that
\be\label{comparaison-poids}
e^{\sqrt{\gamma} \xi}\leq C_\gamma w(\xi)\quad \forall \xi \in \R.
\ee
However, because of the strong degeneracy and singular behavior in the transition zone near $\xi^0_\gamma$,  the constant $C_\gamma$ might be exponentially large.
It follows that
\[
\int_\R u^2 e^{\sqrt{\gamma} \xi}\leq C_\gamma \int_\R u^2 w\leq C_\gamma c_\gamma^{-1} \cD.
\]
Concerning the integral $\int_\R \frac{u^2}{N^\gamma} e^{\sqrt{\gamma} \xi}$, using Proposition \ref{prop:poincare} together with \eqref{comparaison-poids} yields a bound $ C_\gamma c_\gamma^{-1} \cD$ for this term, with a possibly different constant $C_\gamma$. The last term is treated in the same fashion. We obtain
\[
\left| \int_\R \p_x G(u) u W_1\right|\leq C_\gamma \left( \left\| u \frac{\p_x N}{N }\right\|_{L^{\infty}} + \|u\|_{L^{\infty}}\right) \cD \leq \frac{1}{8}\cD,
\]
provided $\|u\|_{L^{\infty}}$ and $\|u\p_x N/N\|_{L^{\infty}}$ are small enough (exponentially small with $\gamma$).

There remains to address the terms containing $W_2$. We have, using Lemma \ref{lem:estim-G} and recalling the expression of the weight $\bar w_0$ from Lemma \ref{lem:w0},
\begin{eqnarray*}
&&\left| \int_\R \p_x G(u) u W_2\right|\\
&\leq & C\gamma^2 \int_\R \frac{|\p_x N|}{N} |u|^3 w + C \gamma \int_\R u^2 |\p_x u| w + \int_\R |u|^3 \frac{w}{N^\gamma}  +  \int_\R \frac{|u|^3(1-N^\gamma)}{|\p_x N| N^{\gamma-1}}w.
\end{eqnarray*}
Using the same arguments as above, we find that the first three terms are bounded by
\[
\left(\gamma^2\left\| u \frac{\p_x N}{N }\right\|_{L^{\infty}} c_\gamma^{-1} + \gamma \|u\|_{L^{\infty}} C_\gamma + \gamma  \left\| u \right\|_{L^{\infty}} c_\gamma^{-1} C_\gamma\right)\cD,
\]
and can be absorbed in $\cD$ under the assumptions of the Lemma, provided $\eta$ is sufficiently small.

There remains the last term, which has an additional singularity in the congested zone because of the $\p_x N$ factor in the denominator (note that in the free zone $\xi>\txg*$, $|\p_x N| \gtrsim N$, so that this singularity can be treated thanks to the weighted Poincaré inequality from Proposition \ref{prop:poincare}.) However this singularity is compensated by the factor $(1-N^\gamma)$ in the numerator. More precisely, we have, for $\xi\leq \txg$ and using Lemma \ref{lem:quantitative-bounds},
\[
\left| \frac{1-N^\gamma}{\p_x N}\right| = \gamma \left|\frac{1-P}{P'}\right|  P^{1-\frac{1}{\gamma}}\lesssim \gamma.
\]
Hence
\[
\int_{x-ct\leq \txg} \frac{|u|^3(1-N^\gamma)}{|\p_x N| N^{\gamma-1}}w \leq C \gamma \|u\|_{L^{\infty}} \int_{x-ct<0} |u|^2 w \leq C \gamma \|u\|_{L^{\infty}} c_\gamma^{-1} \cD.
\]

Gathering all the terms, we obtain the inequality announced in the Lemma.

\end{proof}

Let us now complete the proof of Theorem \ref{thm:stab-Ng}. First, we choose  $h$ so that $h\leq \eta_2/C$, where $\eta_2\in ]0,1[$ is the constant appearing in Lemmas \ref{lem:NL1} and \ref{lem:NL2}, and $C$ is the constant in \eqref{eq:u-max}. Then Lemma \ref{lem:u-max} entails that
\[
\|u\|_\infty + \left\| u \frac{\p_x N}{N}\right\|_\infty \leq \eta_2^\gamma.
\]
It follows from Lemmas  \ref{lem:NL1} and \ref{lem:NL2} that the sum of the two nonlinear terms in the right-hand side of \eqref{eq:u-NL} is bounded by $\int_0^t\cD/2$. Therefore, we obtain for all $t\geq 0$,
\[
\int_\R |u|^2 w + \frac{1}{2}\int_0^t \cD \leq \int_\R |u_0|^2 w.
\]
Letting $t\to \infty$, we obtain the control of the diffusion announced in Theorem \ref{thm:stab-Ng}.
Now, applying the Poincaré inequality from Proposition \ref{prop:poincare} and Lemma \ref{lem:mod-weight}, we get, for all $t\geq 0$,
\[
\int_\R |u|^2 w + \frac{c_\gamma}{2}\int_0^t |u|^2 w \leq \int_\R |u_0|^2 w.
\]
The exponential decay with a rate $c_\gamma/2$ (which we rename $c_\gamma$) follows easily from the Gronwall Lemma.

\section{Proofs of some technical results}{\label{sec:appendix}}

\subsection{Derivation of the equation on $u$}
\label{ssec:eq-u}

In this subsection, we prove that $u$ defined by
\[
u(t,x)=\frac{n(t,x)-N(x-ct)}{\p_x N(x-ct)},
\]
is a solution of \eqref{eq:u-NL}. As in the previous section,  we  omit the dependency of $\gamma$ for simplicity.

First, we recall that $n $ and $N(x-ct)$ are both solutions of \eqref{eq:ng}, in which we rewrite the diffusion term as
\[
\gamma \p_x(n^\gamma \p_x n)= \frac{\gamma}{\gamma+1}\p_{xx} n^{\gamma+1}. 
\]
Recalling the definition of $G(u)$ from \eqref{def:G}, we write (omitting  $x-ct$ in the argument of $N$)
\[
n^{\gamma+1}-N^{\gamma+1}= (\gamma+1) N (n-N) + G(u).
\]
Introducing $\nu(t,x) = n(t,x)-N(x-ct)$, we find that $\nu$ is a solution of
\be
\label{eq:linearized}
\p_t \nu -\gamma \p_x^2(N^\gamma \nu) - \nu \left(1-(\gamma +1) N^\gamma\right) = \frac{\gamma}{\gamma+1}\p_x^2 (G(u)) - G(u).
\ee
 Observe that, from \eqref{eq:Ng}, $\p_x N(x-ct)$ is a (negative) solution of the linearized equation
 \[
 \p_t \p_x N(x-ct) -\gamma \p_x^2(N^\gamma\p_x N(x-ct)) - \left(1-(\gamma +1) N^\gamma\right)\p_x N(x-ct)  = 0.
 \]
 Let us now set $u(t,x):=\nu(t,x)/\p_x N(x-ct)=(n(t,x)-N(x-ct))/\p_x N(x-ct)$ and compute the equation satisfied by $u$. Using the identity
\begin{eqnarray*}
\p_{x}^{2} u&=& \p_x \left( \frac{\p_x \nu}{\p_x N} - \frac{\nu \p_{x}^{2} N }{(\partial_x N)^2}\right)\\
&=& \frac{\p_{x}^{2} \nu}{\p_x N} - \frac{\nu}{(\p_x N)^2}\p^3_{x} N - 2\frac{\p_{x}^{2} N}{\p_x N} \p_x u,
\end{eqnarray*}
we infer that
\be\label{eq:u2}
\p_t u + b \p_x u - a \p_{x}^{2} u=\dfrac{\gamma}{\gamma+1} \dfrac{\partial^2_x G(u)}{N'(x-ct)} - \dfrac{G(u)}{N'(x-ct)} ,
\ee
where $a(t,x)=\bar a(x-ct)$, $b(t,x)=\bar b(x-ct)$ and
\be
\bar b:=-2\gamma \p_x N^\gamma - 2 \gamma N^\gamma \frac{\p_{x}^{2} N}{\p_x N},\quad
\bar a:=\gamma N^\gamma.
\ee

\subsection{Structure of the linearized system: Lemmas \ref{prop:lin} and equality \eqref{eq:u-lin-mod}}
\label{ssec:linearized}

\begin{proof}[Proof of Lemma \ref{prop:lin}]
Multiplying \eqref{eq:u2} by $2 u w_{0}$ and integrating on $\R$, we obtain formally
\[
\frac{d}{dt}\int_{\R}|u|^2w_{0} - \int_{\R} |u|^2\p_t w_{0} + \int_\R 2 u \p_x u (bw_{0} + \p_x (aw_{0})) + 2\int_\R a w_{0}  (\p_x u)^2 =2\int_\R S u w_0.
\]
Integrating by parts the middle term gives
\begin{eqnarray*}
&&\int_\R 2 u \p_x u (bw_{0} + \p_x (aw_{0}))\\
&=& -\int_\R |u|^2 \p_{x}(bw_{0} + \p_x (aw_{0})).
\end{eqnarray*}
Gathering all the terms, we have
\[
\frac{d}{dt}\int_{\R}|u|^2w_{0} - \int_{\R} |u|^2\left[ \p_t w_{0} +\p_{x}(bw_{0} + \p_x (aw_{0}))\right]  + \int_\R a w_{0}  (\p_x u)^2 =2\int_\R S u w_0.
\]
Now, let us look at the term between brackets. As $w_{0} = \bar w_{0}(x-ct)$,
\[\ba
\left[ \p_t w_{0} +\p_{x}(bw_{0} + \p_x (aw_{0}))\right]&=
-c \bar w_{0}' + (\bar b\bar w_{0} + (\bar a\bar w_{0})')',\\
&=\partial_{x}\left((\bar a\bar w_{0})'+(\bar b-c) \bar w_{0})\right),\\
&=0,
\ea
\]
from the definition of $\bar w_{0}$.
This implies
\[
\frac{d}{dt}\int_{\R}|u|^2w_{0} +2 \int_\R a w_{0}  (\p_x u)^2 =2\int_\R S u w_0,
\]
and therefore, integrating with respect to $t$,
\begin{multline*}
   \int_{\R}|u(t,x)|^2w_{0}(t,x) dx + 2\int_{0}^{t}\int_\R a(s,x) w_{0}(s,x)(\p_x u(s,x))^2 dx ds  \\= \int_{\R}|u^{0}(x)|^2w_{0}(0,x) dx + 2 \int_0^t \int_\R S u w_0. 
\end{multline*}
 Note that $w_{0}(0,x)=\bar w_{0}(x)$. 
 Hence we obtain the identity announced in the Lemma.
\end{proof}

\begin{proof}[Proof of equality \eqref{eq:u-lin-mod}]
This proof is very similar to that of Lemma \ref{prop:lin}. Observe first that for $ w= w_{0} \phi$, one has
\begin{eqnarray*}
-[\p_t w + \p_x (bw + \p_{x}(aw))](t,x)&=&\left(c\bar w_{0} \bar\phi  - \bar b \bar w_{0} \bar\phi -(\bar a \bar w_{0} \bar \phi)'\right)'(x-ct),\\
&=& -\left( \bar a \bar w_{0} \bar \phi'\right)'(x-ct).
\end{eqnarray*}
Note that from the definition of $\phi$, $\bar a \bar w_{0} \bar \phi' = -\frac{\delta_{\gamma}}{\sqrt{\gamma}}\exp(\sqrt{\gamma}\xi)$, hence
\begin{equation*}
-[\p_t w + \p_x (bw + \p_{x}(aw))](t,x) = \delta_{\gamma} \exp(\sqrt{\gamma}\xi).
\end{equation*}
Proceeding exactly as in the proof of Lemma \ref{prop:lin}, we obtain
\begin{equation*}
\begin{split}
\frac{d}{dt}\int_{\R}|u|^2w - \int_{\R} |u|^2\left[ \p_t w +\p_{x}(bw + \p_x (aw))\right]  + 2\int_\R a w (\p_x u)^2 &=2\int_\R S u w,\\
\frac{d}{dt}\int_{\R}|u|^2w + \delta_{\gamma} \int_{\R} |u(t,x)|^2\exp(\sqrt{\gamma}(x-cs))dx  + 2\int_\R a w (\p_x u)^2 &=2\int_\R S u w,
\end{split}
\end{equation*}
and therefore, integrating again with respect to $t$, we obtain, for all $t\geq 0$,
\begin{equation*}
\begin{split}
&\int_{\R}|u(t,x)|^{2}w(t,x) dx+
\delta_{\gamma}\int_{\R_{+}}\int_{\R}|u(s,x)|^{2}\exp(\sqrt{\gamma}(x-cs))dx\\
 &+ 2\int_{\R_{+}}\int_\R a(s,x) w(s,x) (\p_x u(s,x))^2 dx ds 
 = \int_{\R}|u^{0}(x)|^{2}\bar w(x) dx+  2 \int_0^t \int_\R S u w .
 \end{split}
\end{equation*}
The Poincaré inequality stated in Lemma \ref{lem:mod-weight} is an easy consequence of Proposition \ref{prop:poincare} and of the equivalence between the weights $\bar w$ and $\bar w_0$.
\end{proof}

\subsection{Properties of the weights $\bar w_0$ and $\bar w$: Lemma \ref{lem:w0} and estimate \eqref{est:delta_gamma} }\label{sec:proof-lem-3.4}

\begin{proof}[Proof of Lemma \ref{lem:w0}]
Let us rewrite equation \eqref{eq:w} as
\[
(\bar a \bar w_0)' + \frac{\bar b-c}{\bar a}(\bar a \bar w_0)=0,
\]
which yields, since $\bar w_0(\xim)=1$,
\[
\bar a \bar w_0(\xi)= \bar a(\xim) \exp\left(\int_{\xim}^\xi \frac{c-\bar b}{\bar a}\right).
\]
We recall that	
	\begin{align*}
	\bar a = \gamma  N^\gamma, \quad 
	\bar b = -2\gamma^2 N^{\gamma-1}\partial_x N - 2\gamma  N^\gamma \dfrac{\partial^2_x N}{\partial_x N},
	\end{align*}
hence
\begin{align*}
-\int_{\xim}^\xi \frac{\bar b(z)}{\bar a (z)} dz
& = 2 \int_{\xim}^\xi \left[\frac{( N^\gamma)'(z)}{ N^\gamma(z)} +  \frac{ N''(z)}{ N'(z)}\right] dz\\
& = 2 \ln \left(\frac{( N(\xi))^\gamma}{( N(\xim))^\gamma}\right) + 2 \ln \left(\frac{| N'(\xi)|}{| N'(\xim)|}\right),
\end{align*}
and therefore
\begin{align}
\bar w_0 (\xi)
& = \dfrac{\bar a(\xim)}{\bar a(\xi)} \exp\left(\int_{\xim}^\xi \dfrac{c}{\bar a}\right) \exp\left(-\int_{\xim}^\xi\dfrac{\bar b}{\bar a}\right) \\
& = \dfrac{1}{ (N(\xim))^{\gamma} ( N'(\xim))^2}  (N(\xi))^\gamma ( N'(\xi))^2\exp\left(\int_{\xim}^\xi \frac{c}{\bar a} dz \right).
\end{align}
Therefore we find the expression announced in Lemma \ref{lem:w0}, with a normalization constant
\[
K:=\dfrac{1}{ (N(\xim))^{\gamma} ( N'(\xim))^2}.
\]
Let us now estimate $K$. We recall that 
$\xim$ is defined in \eqref{def:xim-txg}
Since $N'=\gamma^{-1} P' P^{\frac{1}{\gamma}-1}$, it also follows that
\[
N'(\xim)= \frac{1}{\gamma} P'(\xim)\left(\frac{(c-1)(\gamma+1)}{c^3}\right)^{\frac{1}{2} - \frac{1}{2\gamma}},
\]
and thus
\[
K=\left(\frac{c^3}{(c-1)(\gamma+1)}\right)^{\frac{1}{2} - \frac{1}{\gamma}} \frac{\gamma^2}{P'(\xim)^2}.
\]
The sub- and super-solutions for $P$ (see Proposition \ref{prop:qualitative-Pg-Ng}) entail that $P'(\xim)$ is bounded from above and below. Hence $K$ is of order $\gamma^{3/2}$.

For $\xi \geq \tilde \xi$, we know by~\eqref{encadr-N-free} that there exist $0 < A_1 < A_2 < 1$ (close to $1-c^{-1}$) such that 
\begin{align}\label{eq:behav-N+}
   A_1 e^{-\frac{2}{c}\xi}& \leq  N(\xi) \leq A_2 e^{-\frac{\xi}{c}}.
\end{align}
Moreover, remember that $L_\gamma = \frac{N'}{N} + \frac{1}{c}$ converges uniformly to $0$ on $[\xi^*,+ \infty)$ as $\gamma \to + \infty$ (see the proof of Lemma~\ref{lem:quantitative-bounds}). 
Hence, for any $\eta > 0$, there exists $\gamma_0$ such that for all $\gamma > \gamma_0$, $N'/N \in [-1/c -\eta, -1/c + \eta]$ for all $\xi \in [\xi^*,+ \infty)$.
By~\eqref{eq:behav-N+}, we deduce that
\begin{equation}
\tilde A_1 e^{-\frac{2}{c}\xi} \leq | N'(\xi)| \leq \tilde A_2 e^{-\frac{\xi}{c}},
\end{equation}
with $\tilde A_{1,2} \in (0,1)$.
As a consequence $\bar w_0$ has a double exponential growth as $\xi \to +\infty$:
\begin{align}\label{eq:est-w-pinf}
& K_1 A_1^\gamma \tilde A_1^2 \exp\left(-2\frac{\gamma+2}{c}\xi\right) \exp\left(\dfrac{c^2}{\gamma^2}A_2^{-\gamma}\left(\exp\left(\frac{\gamma}{c}\xi \right)- \exp\left(\frac{\gamma}{c}\xim \right)\right)\right) \nonumber \\
& \quad \leq \bar w_0 \\
& \qquad \leq K_2 A_2^\gamma \tilde A_2^2 \exp\left(-\frac{\gamma+2}{c}\xi\right) \exp\left(\dfrac{c^2}{2\gamma^2}A_1^{-\gamma}\left(\exp\left(\frac{2\gamma}{c}\xi \right)- \exp\left(\frac{2\gamma}{c}\xim \right)\right)\right), \nonumber 
\end{align}
with $K_{1,2} \propto \gamma^{3/2}$. The estimate from the Lemma follows.\\
For $\xi \to -\infty$, we have using~\eqref{cvg-Ng-minf} and denoting $\tilde{\lambda} = \sqrt{1+\frac{c^2}{4\gamma^2}} -\frac{c}{2\gamma}$,
\begin{align*}
\bar w_0
& = K N^\gamma (N')^2 \exp\left(\int_{\xim}^\xi \dfrac{c}{\gamma N^\gamma}\right) 
\underset{\xi \to -\infty}{\sim}  K \tilde\lambda^2 P (1-N)^2 \exp\left(\int_{\xim}^\xi \dfrac{c}{\gamma P}\right) .
\end{align*}
Recalling estimates \eqref{encadr-P-xim}, we get
\[
\frac{C}{\sqrt{\gamma}}(\xi-\xim)\leq \int_{\xim}^\xi \frac{c}{\gamma P}\leq \frac{c}{\gamma}(\xi-\xim),
\]
and, for $|\xi|\leq C$ with $C$ independent of $\gamma$
\[
\frac{C^{-1}}{\gamma} \left(1-\frac{C'}{\sqrt{\gamma}}\right) \exp\left( \xi\right)\leq 1-N \leq \frac{C}{\gamma} \left(1-\frac{C'}{\sqrt{\gamma}}\right) \exp\left(\left(1-\frac{C}{\sqrt{\gamma}}\right) \xi\right).
\]
Gathering the terms, we obtain
\[
C^{-1}\frac{K}{\gamma^2}\exp\left(2\left(1+\frac{C}{\sqrt{\gamma}}\right) \xi\right) \leq \bar w_0 \leq C\frac{K}{\gamma^2}\exp\left(2\left(1-\frac{C}{\sqrt{\gamma}}\right) \xi\right), 
\]
and we deduce the result announced in Lemma~\ref{lem:w0}.

\end{proof}

\begin{proof}[Estimate of the constant $\delta_\gamma$ in $\bar w$]
Let 
\[
\bar \psi := \bar \phi' \bar a \bar w_0
= -\dfrac{\delta_\gamma}{\sqrt{\gamma}}\exp\left(\sqrt{\gamma}\xi\right). 
\]
Using the previous Lemma and recalling~\eqref{eq:behav-N+}, we observe that the double exponential growth of $\bar{w}_0$ dominates the growth in $\psi$ as $\xi \to +\infty$. 
On the other hand, for $\xi \leq -C$, we have
\[
\dfrac{\bar \psi(\xi)}{\bar{a}(\xi)\bar{w}_0(\xi)} 
= -\dfrac{\delta_\gamma}{\sqrt{\gamma}} \dfrac{\exp\left(\sqrt{\gamma}\xi \right)}{\gamma (\bar{N}(\xi))^\gamma\bar w_0(\xi)} 
\geq - C\dfrac{\delta_\gamma}{\gamma^3} \exp \left(\left(\sqrt{\gamma} - 2 \left(1 + \frac{C}{\sqrt{\gamma}}\right)\right)\xi \right) . 
\]
Hence, for $\gamma$ large enough $\sqrt{\gamma} - 3 > 0$ and $\dfrac{\bar \psi(\xi)}{\bar{a}(\xi)\bar{w}_0(\xi)} $ decreases exponentially to $0$ as $\xi \to -\infty$. We conclude then to the integrability of $\dfrac{\bar \psi}{\bar{a}\bar{w}_0}$ on $\R$.

Let us now study the behavior of $\phi$. 
For that purpose, we analyze separately the different regions according to the value of $\xi$.

\begin{itemize}
    \item for $\xi > \tilde\xi$: using the estimates~\eqref{eq:behav-N+}-\eqref{eq:est-w-pinf}, we infer that
    \begin{align*}
   | \phi'(\xi) |
    & = -\dfrac{\bar \psi(\xi)}{\bar{a}(\xi)\bar{w}_0(\xi)} \\
    & \leq C \frac{\delta_\gamma}{\gamma^{3/2}K} A_1^{-2\gamma} \exp\left(\left(\sqrt{\gamma}+  \frac{4(\gamma+1)}{c}\right)\xi\right) 
    \exp\left(-A_2^{-\gamma}\left(\exp\left(\frac{\gamma}{c}\xi\right) - \exp\left(\frac{\gamma}{c}\xim\right) \right) \right)\\
    & \leq C \frac{\delta_\gamma}{\gamma^{3/2}K} A_1^{-2\gamma} \exp\left( 5\frac{\gamma}{c}\xi\right) \exp\left(-A_2^{-\gamma}\left(\exp\left(\frac{\gamma}{c}\xi\right) - \exp\left(\frac{\gamma}{c}\xim\right) \right) \right).
    \end{align*}
    Hence, integrating and using the asymptotic values of $\xim$ and $\tilde \xi$ from Lemma \ref{lem:txg-xim}, we get
    \begin{eqnarray}
    \nonumber\int_{\tilde \xi}^\infty|\bar \phi'| &\leq& C\frac{\delta_\gamma}{\gamma^4} A_1^{-2\gamma} A_2^\gamma \exp\left(-A_2^{-\gamma} \left(\exp\left(\frac{\gamma}{c}\tilde \xi\right) - \exp\left(\frac{\gamma}{c}\xim\right) \right) \right)\\&\leq& C\frac{\delta_\gamma}{\gamma^4} A_1^{-2\gamma} A_2^\gamma \exp(-C A_2^{-\gamma}).\label{phi-1}
    \end{eqnarray}
    %
    \item for the intermediate region $\xi \in [\xi^-, \tilde \xi]$, we write
    \begin{align*}
    \int_{\xi^-}^{\tilde{\xi}}| \bar\phi'(z)| dz
    & = \dfrac{\delta_\gamma}{\sqrt{\gamma}} \int_{\xi^-}^{\tilde{\xi}} \dfrac{e^{\sqrt{\gamma}z}}{\gamma (N(z))^\gamma} 
    \times \dfrac{1}{K (N(z))^\gamma (N'(z))^2 \exp(\int_{\xi^-}^z \frac{c}{\gamma N^\gamma})} dz \\
    & = \dfrac{\delta_\gamma}{\gamma^{-1/2}K}
    \int_{\xi^-}^{\tilde{\xi}} \dfrac{e^{\sqrt{\gamma}z}}{ (P'(z))^2(N(z))^2} \exp\left(-\int_{\xi^-}^z \frac{c}{\gamma N^\gamma}\right)dz.
    \end{align*}
    Now, using Lemma~\ref{lem:txg-xim} and the definition of $\xi^-$, we have in this region
    \[
    |P'(\xi)| \geq \dfrac{C}{\gamma}, \quad
    N(\xi) \geq N(\tilde{\xi}) > 1 - \frac{2}{c}, \quad
    (N(\xi))^\gamma \leq P(\xi^-) = O\left(\dfrac{1}{\sqrt{\gamma}}\right),
    \]
    and thus 
    \begin{align} 
     \int_{\xi^-}^{\tilde{\xi}} |\phi'(z)| dz
    & \leq C \dfrac{\delta_\gamma}{\gamma^{-5/2}K}e^{\sqrt{\gamma}\tilde{\xi}} \int_{\xi^-}^{\tilde{\xi}}  \exp\left(-\dfrac{C}{\gamma^{1/2}}z\right) dz \nonumber \\
    & \leq C \gamma^{3/2} \delta_\gamma \left( \exp\left(-\frac{C\xim}{\sqrt{\gamma}}\right) -\exp\left(-\frac{C\tilde \xi}{\sqrt{\gamma}}\right)\right) \nonumber,\\
    &\leq  C \gamma^{1/2} \delta_\gamma,\label{phi-2}
    \end{align}
    where we have used the fact that $\tilde{\xi} = O(\gamma^{-1})$, $\xim = O(\gamma^{-\frac{1}{2}})$ 
    (cf Lemma~\ref{lem:txg-xim}).
    \item for $\xi < \xim$, we use the fact that
    \[
    \frac{1-P}{P'} \underset{\gamma \to +\infty}{\to} -1 \quad \text{uniformly on}~\R_- ,
	\]
	and the control~\eqref{eq:encadr_P_R-}
	\[
	1- P \geq  \left(1-\frac{1}{\gamma}\right) e^\xi \quad \forall \ \xi < 0,
	\]
	to infer that
	\begin{align}\label{phi-3}
		-\int_{-\infty}^{\xi^-} \phi'(z) dz
		& = \dfrac{\delta_\gamma}{\gamma^{-1/2}K}
    \int_{-\infty}^{\xi^-}  \dfrac{e^{\sqrt{\gamma}z}}{ (P'(z))^2(N(z))^2} \exp\left(\int^{\xi^-}_z \frac{c}{\gamma N^\gamma}\right)dz \nonumber \\
    & \leq C \dfrac{\delta_\gamma}{\gamma^{-1/2}K}
    \int_{-\infty}^{\xi^-}  \dfrac{e^{\sqrt{\gamma}z}}{ (1-P(z))^2} e^{\frac{c}{\sqrt{\gamma}} (\xi^--z)} dz \nonumber \\
    & \leq C \dfrac{\delta_\gamma}{\gamma^{-1/2}K}
    \int_{-\infty}^{\xi^-}  e^{\sqrt{\gamma}z} e^{-2z} e^{\frac{c}{\sqrt{\gamma}} (\xi^--z)} dz \nonumber \\
    & \leq C \dfrac{\delta_\gamma}{\gamma^{-1/2}K}
    \int_{-\infty}^{\xi^-}  e^{\frac{\sqrt{\gamma}}{2}z} dz \nonumber \\
    & \leq C \gamma^{-3/2} \delta_\gamma,
	 \end{align}
\end{itemize}
thanks to the fact that $\xi^- = O(\gamma^{-1/2})$ (cf Lemma~\ref{lem:txg-xim}).
Combining \eqref{phi-1}-\eqref{phi-2}-\eqref{phi-3}, there exists $\delta_0$ bounded away from $0$ such that for $\delta_\gamma = \delta_0 \gamma^{-1/2}$, $\left|\int_\R \phi'\right| \leq 1$.
\end{proof}

\subsection{Proof of the weighted Poincar\'e inequality}\label{proof-poincare}



\begin{proof}[Proof of Proposition \ref{prop:poincare}]
To lighten the notations, we forget in what follows the notation $\bar{\cdot}$ when it is clear that we work with functions of variable $\xi$.
Formally, we have the following inequalities, for any $\rho\in \mathcal C^2(\R)$ and $v\in \mathcal C^1_c(\R)$
\begin{align*}
0 
& \leq \int_{\R} \big(\partial_\xi (v\rho)\big)^2 d\xi \\
& = \int_{\R} \Big[\rho^2 (\partial_\xi v)^2 + v^2(\partial_\xi \rho)^2 d\xi + 2 \int_{\R}u \partial_\xi v \rho \partial_\xi \rho \Big] d\xi \\
& = \int_{\R} \Big[\rho^2 (\partial_\xi v)^2 + v^2(\partial_\xi \rho)^2\Big] d\xi - \int_{\R} v^2 \partial_\xi(\rho \partial_\xi \rho) d\xi \\
& = \int_{\R} \rho^2 (\partial_\xi v)^2 d\xi - \int_{\R} v^2 \rho \partial^2_\xi \rho \ d\xi.
\end{align*}
Note that when $\rho$ is positive and strictly convex, we obtain a Poincaré inequality.
We want to apply this inequality with $\rho:=(\bar a \bar w_0)^{1/2}$.
However, the weight $\rho$ is not convex on $\R$ and we cannot guarantee the sign of the second integral.
Let us compute the derivatives of $\rho$. Using \eqref{formula-w_0}, we have
\[
\rho= -\sqrt{\gamma K} N' N^\gamma \exp \left( \int_{\xim}^\xi \frac{c}{2 \gamma N(\xi)^\gamma}\right),
\]
and thus
\[
\rho'(\xi) 
= \sqrt{\gamma K} \exp\left(\int_{\xim}^\xi \frac{c}{2 \gamma N(\xi)^\gamma}\right) \left[ -(N'(\xi) N(\xi)^\gamma)' - \frac{c}{2\gamma}  N'(\xi) \right].
\]
We recall that
\[
-c N' - \gamma (N' N^\gamma)' = N (1-N^\gamma),
\]
so that
\[
\rho'(\xi) 
= \sqrt{\gamma K} \exp\left(\int_{\xim}^\xi \frac{c}{2 \gamma  N^\gamma(z)}\right) 
\left[ \frac{1}{\gamma} N(\xi) (1- N^\gamma(\xi)) + \frac{c}{2\gamma}  N'(\xi) \right].
\]
Differentiating once again and using the equation on $N$, we get
\begin{align*}
\rho''(\xi) 
& = \sqrt{\frac{K}{\gamma}} \exp\left(\int_{\xim}^\xi \frac{c}{2 \gamma N^\gamma(z)} dz\right)\left[ N' (1-(\gamma+1)N^\gamma) + \frac{c}{2}N'' + \frac{c}{2\gamma N^{\gamma-1}} (1-N^\gamma) + \frac{c^2}{4 \gamma}\frac{N'}{N^\gamma}\right]\\
&= \sqrt{\frac{K}{\gamma}}\exp\left(\int_{\xim}^\xi \frac{c}{2 \gamma N^\gamma(z)}d z \right)\left[ N' (1-(\gamma+1)N^\gamma) \right.\\
& \hspace{2cm}\left. - \frac{c}{2\gamma N^\gamma}\left(c N' + \gamma^2 ( N')^2 N^{\gamma-1} +  N (1- N^\gamma)\right)+ \frac{c}{2\gamma  N^{\gamma-1}} (1- N^\gamma) + \frac{c^2}{4 \gamma}\frac{ N'}{ N^\gamma}\right]	\\
&=-\sqrt{\frac{K}{\gamma}}N'\exp\left(\int_{\xim}^\xi \frac{c}{2 \gamma  N^\gamma(z)}dz \right)\left[  (\gamma+1) N^\gamma + \frac{c^2}{4 \gamma N^\gamma} + c \gamma \frac{ N'}{2 N}-1\right].
\end{align*}
Note that $ \rho'' \geq 0$ provided the term in brackets is non-negative. The bracketed term is a sum of four terms, among which the first two are positive, and the last two are negative. Furthermore,
\[
\frac{\gamma+1}{c} N^\gamma + \frac{c}{4 \gamma N^\gamma} = \left(\sqrt{\frac{\gamma+1}{c} N^\gamma} - \sqrt{\frac{c}{4 \gamma N^\gamma}}\right)^2 + \sqrt{\frac{\gamma+1}{\gamma}} \geq 1,
\]
so that
\begin{equation}\label{eq:d2rho1}
\rho'' \geq  -\sqrt{\frac{K}{\gamma}}N'\exp\left(\int_{\xim}^\xi \frac{c}{2 \gamma N^\gamma}\right)\left[(\gamma+1)\left( 1-\frac{1}{c}\right) N^\gamma + \frac{c(c-1)}{4 \gamma N^\gamma} + \frac{1}{4\gamma} +  c \gamma \frac{N'}{2N}\right].
\end{equation}
Thus the only zone where $\rho''$ is non-positive is the region where the last term in the above bracket is not dominated by others. Decomposing the domain in three zones, we have
\begin{itemize}
	\item $\xi \geq \tilde \xi$: 
	using the notations of Section~\ref{sec:TW-descript}, we have
	\[
	N'(\xi) \geq \tilde Q (N) = - \dfrac{c-1}{4\gamma^2 (N(\xi))^{\gamma-1}},
	\]
	so that
	\[
	c \gamma \frac{N'}{2N} < \frac{c(c-1)}{8 \gamma N^\gamma} \quad \forall \ \xi \geq \tilde \xi.
	\]
	Recalling the expressions of $\rho$ and $\bar w_0$, we infer that for $\xi > \tilde \xi$
	\begin{align*}
	\rho \rho''
	& \geq C \sqrt{\gamma K} \sqrt{\frac{K}{\gamma}} (N')^2 N^\gamma \exp\left(\int_{\xim}^\xi \frac{c}{ \gamma  N^\gamma(z)}dz \right) \dfrac{C}{\gamma N^\gamma}\\
	& \geq \dfrac{C}{\gamma N^\gamma} w_0.
	\end{align*}
	
	\item for  $\xi \leq \xim$, we have $P \geq P(\xim) = \left(\frac{c^3}{(c-1)(\gamma+1)}\right)^{1/2}$ while $P' \in [-c,0]$. Hence, we ensure that
	\[
	-cP' \leq c^2\leq  (\gamma+1)\left(1-\frac{1}{c}\right) P^2 \qquad \forall \ \xi \leq \xim,
	\]  
	and therefore
	\[
	- c \gamma \frac{N'}{2N} \leq  \dfrac{(\gamma+1)}{2} \left( 1-\frac{1}{c}\right) N^\gamma \qquad \forall \ \xi \leq \xim.
	\]
	Let us mention that this inequality is precisely the property that lead us to the normalization \eqref{def:xim-txg}.
		We deduce then
	\begin{align*}
	\rho\rho''
	& \geq  K (N')^2 N^{2\gamma} \exp\left(\int_{\xim}^\xi \frac{c}{ \gamma  N^\gamma(z)}dz \right)  \dfrac{(\gamma+1)}{2} \left( 1-\frac{1}{c}\right)\\
	& \geq C \gamma N^\gamma w_0.
	\end{align*}
	
	\item for intermediate region, i.e., $\xi \in [\xim, \tilde \xi]$, we can always bound the negative contribution as follows
	\[
	\rho (\rho'')_-
	\leq\frac{c\gamma}{2} K |N'|^3 \exp\left(\int_{\xim}^\xi \frac{c}{ \gamma  N^\gamma(z)}dz \right) N^{\gamma-1} 
		\leq C_\gamma e^{\sqrt{\gamma}\xi},
	\]
	for $\xi \in (\xim, \tilde \xi)$, where
	\[
	C_\gamma \leq C \gamma^{5/2} \|N'\|_{L^\infty(\xim,\tilde \xi)}^3 \exp\left(\int_{\xi^-}^{\tilde \xi} \frac{1}{\gamma N^\gamma}\right).
	\]
	Let us now evaluate the integral in the argument of the exponential.
	Using Lemma \ref{lem:txg-xim}, we  recall that $|P'|=\gamma |N'| N^{\gamma-1}\geq C\gamma^{-1}$ on $(\xi^-, \tilde \xi)$.
Hence
	\begin{eqnarray*}
	\int_{\xi^-}^{\tilde \xi} \frac{1}{\gamma N^\gamma}&=&\int_{\xi^-}^{\tilde \xi} \frac{|N'|}{\gamma |N'| N^\gamma}\\
	&\leq & C\gamma \int_{\xi^-}^{\tilde \xi} \frac{|N'|}{N}\\
	&\leq & C \gamma \ln \left(\frac{N(\xi^-)}{N(\tilde \xi)}\right) \leq C \gamma.
	\end{eqnarray*}
	Thus $C_\gamma \leq C^\gamma$ for some constant $C>1$ independent of $\gamma$.
	
\end{itemize} 	
Gathering all the terms, we obtain
\[
\int_{\xi\leq \xim} v^2 \rho \rho'' + \int_{\xi\geq \tilde \xi} v^2 \rho \rho'' \leq \int_\R (\p_\xi v)^2 \bar a \bar w_0 + \int_{\xim}^{\tilde \xi} v^2 \rho (\rho'')_-.
\]
Replacing $\rho\rho''$ by their lower bounds on $(-\infty, 0)$ and on $(\tilde \xi, + \infty)$, we obtain the inequality announced in the Proposition.

\end{proof}

\section*{Acknowledgements}
This project has received funding from the European Research Council (ERC) under the European Union's Horizon 2020 research and innovation program Grant agreement No 637653, project BLOC ``Mathematical Study of Boundary Layers in Oceanic Motion''. This work was supported by the SingFlows and CRISIS projects, grants ANR-18-CE40-0027 and ANR-20-CE40-0020-01 of the French National Research Agency (ANR).
A.-L. D. acknowledges the support of the Institut Universitaire de France.
This material is based upon work supported by the National Science Foundation under Grant No. DMS-1928930 while the authors participated in a program hosted by the Mathematical Sciences Research Institute in Berkeley, California, during the Spring 2021 semester.

The authors would like to thank Francois Hamel for pointing out reference~\cite{gilding2005fisher}.

\bibliography{ref_hs}
\end{document}